\documentclass[a4paper]{amsart}

\pdfoutput=1

\usepackage[utf8]{inputenc}
\usepackage[T1]{fontenc}
\usepackage{lmodern}
\usepackage{amsthm, amssymb, amsmath, amsfonts, mathrsfs,stmaryrd}
\usepackage[colorlinks=true, pdfstartview=FitV, linkcolor=blue, citecolor=blue, urlcolor=blue,pagebackref=false]{hyperref}
\usepackage{bbm}

\usepackage{color}

\usepackage{microtype}

\newtheorem{thm}{Theorem}[section]
\newtheorem{prop}[thm]{Proposition}
\newtheorem{lem}[thm]{Lemma}
\newtheorem{cor}[thm]{Corollary}
\theoremstyle{remark}
\newtheorem{rem}[thm]{Remark}
\theoremstyle{definition}
\newtheorem{defi}[thm]{Definition}

\renewcommand{\le}{\leqslant}
\renewcommand{\ge}{\geqslant}
\renewcommand{\subset}{\subseteq}
\newcommand{\mcl}{\mathcal}
\newcommand{\msf}{\mathsf}

\newcommand{\msc}{\mathscr}
\newcommand{\N}{\mathbb{N}}

\newcommand{\Ll}{\left}
\newcommand{\Rr}{\right}

\newcommand{\1}{\mathbf{1}}
\newcommand{\R}{\mathbb{R}}
\newcommand{\Rd}{{\mathbb{R}^d}}
\newcommand{\C}{\mathcal{C}}

\newcommand{\Z}{\mathbb{Z}}
\renewcommand{\P}{\mathbb{P}}
\newcommand{\ov}{\overline}
\renewcommand{\bar}{\overline}

\renewcommand{\hat}{\widehat}
\newcommand{\td}{\tilde}
\newcommand{\eps}{\varepsilon}
\renewcommand{\d}{{\mathrm{d}}}
\newcommand{\dr}{\partial}

\newcommand{\al}{\alpha}
\newcommand{\be}{\beta}

\newcommand{\E}{\mathbb{E}}

\newcommand{\B}{\mathcal{B}}
\newcommand{\Bb}{{\B^{\al}_{p,q}}}

\newcommand{\Bl}[1]{\B^{#1,\, \mathrm{loc}}}
\newcommand{\Bbl}{\B^{\al,\, \mathrm{loc}}_{p,q}}
\newcommand{\Cl}{\mcl{C}^{\al}_{\mathrm{loc}}}

\newcommand{\n}{_{n,x}}

\renewcommand{\ni}{_{n,x}^{(i)}}

\newcommand{\nkold}{_{n,K,p}}

\newcommand{\nzkold}{_{k,K,p}}

\newcommand{\Ln}{\Lambda_n}
\newcommand{\psii}{\psi^{(i)}}
\newcommand{\V}{\msc{V}}
\newcommand{\W}{\msc{W}}

\DeclareMathOperator{\supp}{Supp}

\numberwithin{equation}{section}

\title[Tightness criterion and the Ising model]{A tightness criterion for random fields, with application to the Ising model}
\author{Marco Furlan, Jean-Christophe Mourrat}

\address[Marco Furlan]{Ceremade, PSL - Université Paris-Dauphine, Paris, France}

\address[Jean-Christophe Mourrat]{Ecole normale supérieure de Lyon, CNRS, Lyon, France}



\begin{document}

\begin{abstract}
We present a criterion for a family of random distributions to be tight in local H\"older and Besov spaces of possibly negative regularity on general domains. We then apply this criterion to find the sharp regularity of the magnetization field of the two-dimensional Ising model at criticality, answering a question of \cite{cgn1}. 




\end{abstract}
\maketitle
%
%
%
%
%
%
%
%
\section{Introduction}
\label{s:intro}

The main goal of this paper is to provide a tightness criterion in local H\"older and Besov spaces of negative regularity. Roughly speaking, for $\al <0$, a distribution $f$ on $\Rd$ is $\al$-Hölder regular if for every $x \in \Rd$ and every smooth, compactly supported test function $\varphi$, we have
\begin{equation}
\label{e:behav}
\lambda^{-d} \langle f,\varphi(\lambda^{-1} ( \, \cdot \, - x))\rangle \lesssim \lambda^{\al} \qquad (\lambda \to 0).
\end{equation}

Random objects taking values in distribution spaces are of interest in several areas of probability theory. The spaces considered here are close to those introduced in \cite{struc} in the context of non-linear stochastic PDE's. Another case of recent interest is the scaling limit of the critical two-dimensional Ising model, see \cite{cgn1,chi}. Fluctuations in homogenization of PDE's with random coefficients are also described by random distributions resembling the Gaussian free field, see \cite{correl, fluct,fluct2,akm3, akmbook}. More generally, the class of random objects whose scaling limit is the Gaussian free field is wide, see for instance \cite{nadspe, gos, bisspo} for the $\nabla \varphi$ random interface model, \cite{ke} for random domino tilings, or \cite{lebserf,bauer} for Coulomb gases. 

As in \cite{struc}, we wish to devise spaces where \eqref{e:behav} holds \emph{locally} uniformly over~$x$. Such spaces can be thought of as local Besov spaces. We also wish to allow for distributions that are defined on a domain $U \subset \R^d$, but not necessarily on the full space $\R^d$. Besov spaces defined on domains of $\R^d$ have already been considered, see e.g.\ \cite[Section 1.11]{tri3} and the references therein. In the standard definition, a distribution $f$ belongs to the Besov space on $U$ if and only if there exists a distribution $g$ in the Besov space on $\R^d$ (with same exponents) such that $g_{| U} = f$; the infimum of the norm of $g$ over all admissible $g$'s then provides with a norm for the Besov space on $U$. 

In applications to the problems of probability theory mentioned above, this definition is often too stringent. Consider the case of homogenization. Let $u_\eps$ be the solution to a Dirichlet problem on $U \subset \R^d$, for a divergence-form operator with random coefficients varying on scale $\eps \to 0$. While $\eps^{-\frac d 2} (u_\eps - \E[u_\eps])$ is expected to converge to a random field in the bulk of the domain, this is most likely not the case close to the boundary: a comparably very large boundary layer is expected to be present. This boundary layer should become asymptotically thinner and thinner as $\eps \to 0$, but should nevertheless prevent convergence to happen in a function space such as the one alluded to above.

As a consequence, we will define local Besov spaces that are very tolerant to bad behavior close to the boundary. In short, we take the inductive limit of Besov spaces over compact subsets of the domain. Even when $U = \R^d$, the space thus defined will be stricty larger than the usual Besov space on $\R^d$, because of its locality. This locality is convenient for instance when handling stationary processes.

While we did not find previous works where such spaces appear, readers familiar with Besov spaces will not be surprised by the results presented here.  On the other hand, we hope that probabilists will appreciate to find here a tightness criterion that is very convenient to work with. In order to convince the reader of the latter, we now state a particular case of our main tightness result, Theorem~\ref{t:tight}, when the domain $U$ is the whole space $\R^d$. For each $\al \in \R$, we define a function space $\Cl(\Rd)$ of distributions with ``local $\al$-H\"older regularity'', and for any given $r \ge  | \al|$, we identify a \emph{finite} family of compactly supported functions $\phi$, $(\psii)_{1 \le i < 2^d}$ of class~$C^r$ such that the following holds.

\begin{thm}
\label{t.easy.tightness}
Let $(f_m)_{m \in \N}$ be a family of random linear forms on $C^r_c(\Rd)$, let $1 \le p < \infty$ and let $\be \in \R$ be such that $|\be| < r$. Assume that there exists a constant $C < \infty$ such that for every $m \in \N$, the following two statements hold:
	\begin{equation}
		\label{e:intro-tight1}
		\sup_{x \in \Rd} \E \Ll[ \Ll| \langle f_m,\phi(\, \cdot \, - x) \rangle \Rr|^p  \Rr]^{1/p} \le C \; ;
	\end{equation}
	and, for every $i \in \{1,\ldots, 2^d-1\}$ and $n \in \N$,
	\begin{equation}
		\label{e:intro-tight2}
		\sup_{x \in \Rd} 2^{dn} \ \E \Ll[ \Ll| \langle f_m,\psii(2^n(\, \cdot \, - x)) \rangle \Rr|^p  \Rr]^{1/p} \le C \, 2^{-n \be}.
	\end{equation}
	Then the family $(f_m)$ is tight in $\Cl(\Rd)$ for every $\al < \be - \frac{d}{p}$.
\end{thm}
Note that the assumption in Theorem~\ref{t.easy.tightness} simplifies when the field under consideration is stationary, since the suprema in \eqref{e:intro-tight1} and \eqref{e:intro-tight2} can be removed. Although we are primarily motivated by applications of this result for negative exponents of regularity, the statements we prove are insensitive to the sign of this exponent. Naturally, such tightness statements can then be lifted to statements of convergence in $\Cl(\R^d)$ provided that one verifies that the sequence $(f_m)$ has a unique possible limit point (and the latter can be accomplished by checking that for each test function $\chi \in C^\infty_c(\Rd)$ the random variable $\langle f_m, \chi \rangle$ converges in law as $m$ tends to infinity). 

When $\al < 0$, the definition of the space $\C^\al$ is easy to state and in agreement with the intuition of \eqref{e:behav}, see Definition~\ref{def.Besov} below. For $\al \in (0,1)$, the space $\C^\al$ is (the separable version of) the space of $\al$-H\"older regular functions. For any $\al \in \R$, the space $\C^\al$ is the Besov space with regularity index $\al$ and integrability exponents $\infty, \infty$, which we denote by $\B^\al_{\infty,\infty}$. 
The assumption in Theorem~\ref{t.easy.tightness} is sufficient to establish tightness in $\Bl{\al}_{p,q}(\R^d)$ for every $\al < \be$ and $q \in [1,\infty]$. A variant of the argument also provides for a result in the spirit of the Kolmogorov continuity theorem, see Proposition~\ref{p:kolmotype} below.

The proof of Theorem~\ref{t.easy.tightness} relies on showing that the family $f_m$ belongs with high probability to a bounded set in $\mcl{C}^{\be - d/p}_{\mathrm{loc}}(\R^d)$, and exploiting the well-known compact embedding result of Proposition~\ref{p:compact} to obtain tightness for $\al < \be - \frac{d}{p}$. The more general Theorem~\ref{t:tight} for arbitrary Besov spaces $\Bl{\al}_{p,q}(U)$ follows along the same lines, once we come up with a working definition of $\Bl{\al}_{p,q}(U)$ on an arbitrary domain $U \subset \R^d$. This is done in Proposition~\ref{p:loc-wave} using the concept of spanning sequence introduced in Definition~\ref{def:span}.

The functions $\phi$, $(\psii)_{1 \le i < 2^d}$ are chosen as wavelets with compact support. We found it interesting to distinguish the treatment of $\C^\al$-type spaces from the more general $\B^\al_{p,q}$ spaces. Besides allowing for a simpler definition, the $\C^\al$-type spaces indeed enable us to give a fully self-contained proof of Theorem~\ref{t.easy.tightness}, save for the existence of wavelets with compact support which of course we do not reprove. We borrow more facts from the literature on function spaces to prove the tightness criterion in general Besov spaces.

\smallskip

We then apply the tightness criterion of Theorem~\ref{t:tight} to study the magnetization field of the two-dimensional Ising model at the critical temperature. Let $U \subset \R^2$ be an open set, and for $a > 0$, let $U_a := U \cap (a \Z^2)$. Denote by $(\sigma_y)_{y \in U_a}$ the Ising spin system at the critical temperature, with, say, $+$ boundary condition, and define the magnetization field
\begin{equation} 
\label{e.def.magnetization-delta}
\hat{\Phi}_a := a^{\frac {15}{8}} \sum_{y \in U_a} \sigma_y \, \delta_y ,
\end{equation}
where $\delta_y$ is the Dirac mass at $y$. Dirac masses do not belong to $\B^{-1}_{2,2}(U)$, and thus prevent the family $(\hat{\Phi}_a)_{a \in (0,1]}$ from being tight in this space.
Following \cite{cgn1}, we will thus prefer to work with the piecewise constant random field
\begin{equation}\label{e.def.magnetization}
\Phi_a := a^{-\frac 1 8} \sum_{y \in U_a} \sigma_y \, \mathbbm{1}_{S_a(y)},
\end{equation}
where $S_a(y)$ is the square centered at $y$ of side length $a$. 
We note however that the set of limit points of $(\hat{\Phi}_a)_{a \in (0,1]}$ and of $(\Phi_a)_{a \in (0,1]}$ coincide. Indeed, one can check using Definition~\ref{def.Besov} that the difference $\Phi_a - \hat{\Phi}_a$ converges to zero almost surely in, say, $\mcl{C}^{\alpha}_{\mathrm{loc}}(U)$, for every $\alpha < -3$. 

In \cite{cgn1}, the authors showed that for $U = [0,1]^2$ and every $\eps > 0$, the family $(\Phi_a)_{a \in (0,1]}$ is tight in $\B^{-1-\eps}_{2,2}(U)$\footnote{see for instance \cite[Definition~2.68]{bcd} or \cite{belo} for the identification between the Sobolev spaces used in \cite{cgn1} and the spaces $\B^{\al}_{2,2}$ we use in the present paper.}, and proceeded to discuss similar results in more general domains. 
They asked in which precise function spaces the family $(\Phi_a)$ is tight.

%
Using the Onsager correlation bounds and the tightness criterion for general domains of Theorem~\ref{t:tight}, we prove the following result. 
\begin{thm}
\label{t.Ising}
Fix an open set $U \subset \R^2$. For every $\eps > 0$ and $p,q \in [1,\infty]$, the family of Ising magnetization fields $(\Phi_a)_{a \in (0,1]}$ on $U$ is tight in $\Bl{-\frac 1 8 - \eps}_{p,q}(U)$. 
\end{thm}
We also prove that the previous result is essentially sharp, when $U = \R^2$.
\begin{thm}  
\label{t.Ising.converse}
Let $\eps > 0$ and $p , q \in [1,\infty]$. If $\Phi$ is a limit point of the family of Ising magnetization fields $(\Phi_a)_{a \in (0,1]}$ on $\R^2$, then 
$\Phi \notin \Bl{-\frac 1 8 + \eps}_{p,q}(\R^2)$ with positive probability. In particular, the family $(\Phi_a)_{a \in (0,1]}$ is \emph{not} tight in $\Bl{-\frac 1 8 + \eps}_{p,q}(\R^2)$.
\end{thm}
It was shown recently that there exists a unique limit point to the family $(\Phi_a)_{a \in (0,1]}$, see \cite{cgn1,chi}. Theorem~\ref{t.Ising.converse} makes it clear that this limit is singular (even on compact subsets) with respect to every $P(\varphi)$ Euclidean field theory, since the latter fields take values in $\Bl{-\eps}_{p,q}(\R^2)$ for every $\eps > 0$ and $p,q \in [1,\infty]$.

\smallskip

The paper is organized as follows. In Section~\ref{s:tightcrit}, we review some properties of wavelets and Besov spaces on $\R^d$, define local Besov spaces, and state and prove the tightness criterion in Theorem~\ref{t:tight}, which is a generalization of Theorem~\ref{t.easy.tightness} above. We also provide a version of Kolmogorov's continuity theorem for local Besov spaces in Proposition~\ref{p:kolmotype}. We then turn to the Ising model in Section~\ref{s:isingapp}. After recalling some classical facts about this model, we prove Theorems~\ref{t.Ising} and \ref{t.Ising.converse}.
Appendix~\ref{a:notselfcont} contains some functional analysis results which we needed to prove the tightness criterion in the general Besov space $\Bl{\al}_{p,q}(U)$.

\section{Tightness criterion} \label{s:tightcrit}

We begin by introducing some general notation.  If $u = (u_n)_{n \in I}$ is a family of real numbers indexed by a countable set $I$, and $p \in [1,\infty]$, we write
$$
\|u\|_{\ell^p} = \Ll( \sum_{n \in I} |u_n|^p \Rr)^{1/p},
$$
with the usual interpretation as a supremum when $p = \infty$. We write $B(x,R)$ for the open Euclidean ball centred at $x$ and of radius $R$. For every open set $U \subset \R^d$ and $r \in \N \cup \{\infty\}$, we write $C^r(U)$ to denote the set of $r$ times continuously differentiable functions on $U$, and $C^r_c(U)$ the subset of $C^r(U)$ of functions with compact support. We simply write $C^r$ and $C^r_c$ for $C^r(\R^d)$ and $C^r_c(\R^d)$ respectively. For $f \in C^r$, we write
\begin{equation*}  
\|f\|_{C^r} := \sum_{|i| \le r} \|\dr_i f\|_{L^\infty},
\end{equation*}
where the sum is over multi-indices $i \in \N^d$.

We define the H\"older space of exponent $\al < 0$ very similarly to \cite[Definition~3.7]{struc}.
\begin{defi}[Besov-Hölder spaces]
\label{def.Besov}
Let $\al < 0$, $r_0 := - \lfloor \al \rfloor$, and 
\begin{equation*}  
\mathscr{B}^{r_0} := \{ \eta \in C^{r_0} \ : \ \|\eta\|_{C^{r_0}} \le 1 \text{ and } \supp \eta \subset B(0,1) \}.
\end{equation*}
For every $f \in C^\infty_c$, denote
\begin{equation}
\label{e:def.Cal}
\|f\|_{\C^\al} := \sup_{\lambda \in (0,1]} \, \sup_{x \in \Rd} \, \sup_{\eta \in \mathscr{B}^{r_0}}  \lambda^{-\al} \int_\Rd f \, \lambda^{-d} \eta \Ll( \frac{\cdot - x}{\lambda} \Rr) .
\end{equation}
The \emph{H\"older space} $\C^\al$ is the completion of $C^\infty_c$ with respect to the norm $\|\cdot\|_{\C^\al}$. For every open set $U \subset \Rd$, the \emph{local H\"older space} $\Cl(U)$ is the completion of $C^\infty_c$ with respect to the family of seminorms
\begin{equation*}  
f \mapsto \|\chi f\|_{\C^\al},
\end{equation*}
where $\chi$ ranges in $C^\infty_c(U)$.
\end{defi}
\begin{rem}
\label{r.embed}
By definition, an element of $\C^\al$ defines a continuous mapping on 
\begin{equation*}  
\{ \eta(\cdot - x) \in C^{r_0} \ : \ x \in \Rd,  \ \|\eta\|_{C^{r_0}} \le 1 \text{ and } \supp \eta \subset B(0,1) \}
\end{equation*}
and taking values in $\R$. It is straightforward to extend this mapping to a linear form on $C^{r_0}_c$. In particular, we may and will think of $\C^\al$ as a subset of the dual of $C^\infty_c$. Similarly, the space $\Cl(U)$ can be seen as a subset of the dual of $C^\infty_c(U)$.
\end{rem}
\begin{rem}
\label{r.separ}
Our definition of $\C^\al$ (and similarly for $\Cl$) departs slightly from the more common one consisting of considering all distributions $f$ such that $\|f\|_{\C^\al}$ is finite. The present definition has the advantage of making the space $\C^\al$ separable.
\end{rem}
\begin{rem}
As will be seen shortly, the topology of $\Cl$ is metrisable.
\end{rem}

The gist of the tightness criterion we want to prove is that it suffices to check a condition of the form of \eqref{e:def.Cal} for a \emph{finite} number of test functions. As announced in the introduction, these test functions are chosen as the basis of a wavelet analysis. We now recall this notion.




\begin{defi} \label{def:multi}
A \emph{multiresolution analysis} of $L^2(\R^d)$ is an increasing sequence $(V_n)_{n \in \Z}$ of subspaces of $L^2(\R^d)$, together with a \emph{scaling function} $\phi \in L^2(\R^d)$, such that 
\begin{itemize}
\item $\bigcup_{n \in \Z} V_n$ is dense in $L^2(\R^d)$, $ \quad \bigcap_{n \in \Z} V_n = \{0\}$;
\item $f \in V_n$ if and only if $f(2^{-n} \,  \cdot \, ) \in V_0$;
\item $(\phi(\, \cdot \,  - k))_{k \in \Z^d}$ is an orthonormal basis of $V_0$.
\end{itemize}
\end{defi}

\begin{defi}\label{def:reg-multi}

A multiresolution analysis is called $r$-\emph{regular} ($r \in \N$) if its scaling function $\phi$ can be chosen in such a way that 
$$
| \partial^k \phi (x) | \leqslant C_m (1+ |x |)^{-m}
$$
for every integer $m $ and for every multi-index $k \in \N^d$ with $| k | \leqslant r$.
\end{defi}

While a given sequence $(V_n)$ can be associated with several different scaling functions to form a multiresolution analysis, a multiresolution analysis is entirely determined by the knowledge of its scaling function. We denote by $W_n$ the orthogonal complement of $V_n$ in $V_{n+1}$.

\begin{thm}[compactly supported wavelets]
	\label{t:compact_wave}
	For every positive integer $r$, there exist $\phi$, $(\psii)_{1 \le i < 2^d}$ such that 
	\begin{itemize}
		\item $\phi, (\psii)_{i < 2^d}$ all belong to $C^r_c$;
		\item $\phi$ is the scaling function of a multiresolution analysis $(V_n)$;
		\item $(\psii(\, \cdot \, - k))_{i < 2^d, k \in \Z^d}$ is an orthonormal basis of $W_0$.
	\end{itemize}
	
\end{thm}
This result is due to \cite{dau} (see also e.g.\ \cite[Chapter~6]{pin}). We recall that a wavelet basis on $\Rd$ can be constructed from one on $\R$ by taking products of wavelet functions for each coordinate. We also recall from \cite[Theorem~2.6.4]{mey} that for every multi-index $\be \in \N^d$ such that $|\be| < r$ and every $i < 2^d$, we have
\begin{equation}
\label{e:vanish.moment}
\int x^\be \psii(x) \, \d x = 0.
\end{equation}
Except for Theorem~\ref{t:compact_wave} and \eqref{e:vanish.moment}, we will give a self-contained proof of the tightness criterion in $\Cl(U)$.
From now on, we fix both $r \in \N$ and a wavelet basis  $\phi, (\psii)_{i < 2^d} \in C^r_c$, as obtained with Theorem~\ref{t:compact_wave}.
%
Let $R$ be such that
\begin{equation}
	\label{e:def:R}
	\supp \phi \subset B(0,R), \qquad \supp \psii \subset B(0,R) \quad (i < 2^d).
\end{equation}
For any $n \in \Z$ and $x \in \R^d$, if we define
\begin{equation}
	\label{e:def:scaling}
\phi\n(y) := 2^{dn/2} \, \phi(2^{n}(y-x))
\end{equation}
and $\Ln = \Z^d/2^n$, then $(\phi\n)_{x \in \Ln}$ is an orthonormal basis of $V_n$. Similarly, we define 
\begin{equation*}
\psi\ni(y) := 2^{dn/2} \, \psi^{(i)}(2^{n}(y-x)),
\end{equation*}
so that $(\psii\n)_{i < 2^d, x \in \Ln, n \in \Z}$ is an orthonormal basis of $L^2(\R^d)$. For $f \in L^2(\R^d)$, we set
\begin{equation}
\label{e:def:vw}
v\n f  := (f, \phi\n), \qquad w\ni f := (f,\psi\ni),
\end{equation}
where $(\, \cdot \, , \, \cdot \,)$ is the scalar product of $L^2(\R^d)$. 
Denoting by $\V_n$ and $\W_n$ the orthogonal projections on $V_n$, $W_n$ respectively, we have
\begin{equation}
\label{e:def:VW}
\V_n f = \sum_{x \in \Ln} v\n(f) \, \phi\n, \qquad \W_n f = \sum_{{i < 2^d, x \in \Ln}} w\ni (f) \, \psii\n,
\end{equation}
and for every $k \in \Z$,
\begin{equation}
\label{e:decompf}
f = \V_{k} f + \sum_{n = k}^{+\infty} \W_n f 
\end{equation}
in $L^2(\R^d)$.
 
\begin{defi}[Besov spaces]
	\label{def:Besov:wave}
	Let $\alpha \in \R$, $|\alpha| < r$ and $p,q \in [1,\infty]$. The \emph{Besov space} $\B^{\alpha}_{p,q}$ is the completion of $C^\infty_c$ with respect to the norm
	\begin{equation}
	\label{e:def:Besov}
	\|f\|_{\B^{\alpha}_{p,q}} := 
	\|\V_0 f\|_{L^p} + \Ll\| \Ll( 2^{\al n} \|\W_n f\|_{L^p}  \Rr)_{n \in \N}  \Rr\|_{\ell^q}.
	\end{equation}
	The \emph{local Besov space} $\Bbl(U)$ is the completion of $C^\infty(U)$ with respect to the family of semi-norms 
	$$
	f \mapsto \|\chi f\|_{\Bb}
	$$
	indexed by $\chi \in C^\infty_c(U)$.
	%
	
\end{defi}

\begin{rem}
This characterization of Besov spaces is the most useful for us to obtain a tightness result for local domains, as stressed in the Introduction, due to its projection on compactly supported wavelets. However, in Appendix~\ref{a:notselfcont} we outline a proof of the equivalence between Definition~\ref{def.Besov} and one based on the Littlewood-Paley decomposition (see Definition~\ref{a:def:litpaley}). This second definition is the one used in \cite{bcd} to prove a large number of results on the spaces $\B^{\alpha}_{p,q} $, including their relation with other well-known function spaces.
\end{rem}

\begin{rem}
\label{r:separ}
Similarly to the observation of Remark~\ref{r.separ}, our definition of $\Bb$ departs slightly from the usual one, which consists in considering the set of distributions such that \eqref{e:def:Besov} is finite. The two definitions coincide only when both $p$ and $q$ are finite. The present definition has the advantage of making the space separable in every case, by taking the closure of a family of smooth compactly supported functions. On the other hand, for $\al \in (0,1)$, the more standard definition of the space $\B^{\alpha}_{\infty,\infty}$ would coincide with the Hölder space of regularity $\alpha$ (see Appendix~\ref{a:notselfcont} and \cite{bcd}), which is not separable. 
\end{rem}

\begin{rem}\label{r:def:indep}
One can check that the space $\Bb$ of Definition~\ref{def:Besov:wave} does not depend on the choice of the multiresolution analysis, in the sense that for any $r > | \alpha |$, any different $r$-regular multiresolution analysis yields an equivalent norm (see Proposition~\ref{a:p:equiv} of the appendix). In this section, we recall that we fix $r \in \N$, and consider Besov spaces $\Bb$ with $\alpha \in \R$, $| \alpha | < r$.
\end{rem}

\begin{rem}
\label{r:easy.embed}
It is clear that if $\al_1 \le \al_2 \in \R$ and $q_1 \ge q_2 \in [1,\infty]$, then 
\begin{equation*}  
\|f\|_{\B^{\al_1}_{p,q_1}} \le C \|f\|_{\B^{\al_2}_{p,q_2}},
\end{equation*}
where $C$ is independent of $f \in C^\infty_c$. In particular, the space $\B^{\al_2}_{p,q_2}$ is continuously embedded in $\B^{\al_1}_{p,q_1}$. Similarly, for $p_1 \le p_2$ and for a given $\chi \in C^\infty_c$, there exists a constant $C < \infty$ such that for every $f \in C^\infty_c$,
\begin{equation*}  
\|\chi f\|_{\B^{\al}_{p_1,q}} \le C \|\chi f\|_{\B^{\al}_{p_2,q}}.
\end{equation*}
Indeed, this is a consequence of Jensen's inequality and the fact that for each $n \in \N$, the support of $\W_n(\chi f)$ is contained in the bounded set $2R + \supp \chi$. Hence, the space $\Bl{\al}_{p_2,q}(U)$ is continuously embedded in $\Bl{\al}_{p_1,q}(U)$.
\end{rem}
\begin{rem}  
A different notion of H\"older space on a domain, encoding more precise weighted information on the size of the distribution as one gets closer and closer to the boundary of the domain, has been introduced in the very recent work \cite{gh}. 
\end{rem}

The finiteness of $\|f\|_\Bb$ can be expressed in terms of the magnitude of the coefficients $v\n(f)$ and $w\ni(f)$. 
\begin{prop}[Besov spaces via wavelet coefficients]
	\label{p:charact-coefs}
	For every $p \in [1,\infty]$, there exists $C  \in (0,\infty)$ such that for every $f \in C^\infty_c$ and every $n \in \Z$,
\begin{equation}
	\label{e:comp_coeff}
	C^{-1} \  \|\V_n f\|_{L^p} \le 2^{dn \Ll( \frac12 - \frac1p \Rr) } \Ll\| \Ll( v\n f\Rr)_{x \in \Ln} \Rr\|_{\ell^p} \le C \ \|\V_n f\|_{L^p},
\end{equation}
\begin{equation}
	\label{e:comp_coeff-2}
	C^{-1} \  \|\W_n f\|_{L^p} \le 2^{dn \Ll( \frac12 - \frac1p \Rr) } \Ll\| \Ll( w\ni f\Rr)_{i < 2^d, x \in \Ln} \Rr\|_{\ell^p} \le C \ \|\W_n f\|_{L^p}.
\end{equation}
\end{prop}
\begin{proof} 
We will prove only \eqref{e:comp_coeff} in detail, since \eqref{e:comp_coeff-2} follows in the same way. (See also \cite[Proposition~6.10.7]{mey}.)
Recalling \eqref{e:def:R}, we have $\supp \phi_{n,x} \subseteq B(x,2^{-n}R)$ and thus, for every $y \in \R$,
\begin{equation}
\V_n f (y) = \sum_{x \in \Lambda_n, x \in B(y,2^{-n}R)}v_{n,x}(f) \phi_{n,x}(y).
\end{equation}
Let $p<+\infty$. 
Since the sum $\sum_{x \in \Lambda_n, x \in B(y,2^{-n}R)}$ is finite uniformly over $n$, we can use Jensen's inequality to obtain:
\begin{align*}
\| \V_n f \|_{L^p}^p  = & \int \left| \sum_{x \in \Lambda_n, x \in B(y,2^{-n}R)}v_{n,x}(f) \phi_{n,x}(y) \right|^p \mathrm{d}y \\
\lesssim & \int  \sum_{x \in \Lambda_n, x \in B(y,2^{-n}R)} \left| v_{n,x}(f) \phi_{n,x}(y) \right|^p \mathrm{d}y \\
\lesssim & \sum_{x \in \Lambda_n} \left| v_{n,x}(f) \right|^p \int_{B(x,2^{-n} R)} \left| \phi_{n,x}(y) \right|^p \mathrm{d}y \\
\lesssim  & \Ll\| \Ll( v\n f\Rr)_{x \in \Ln} \Rr\|_{\ell^p}^p \, \|\phi_{n,0}\|_{L^p}^p .
\end{align*}
The leftmost inequality of \eqref{e:comp_coeff} follows from the scaling properties of $\phi_{n,0}$, namely:
\begin{equation} \label{e:phiscaling}
\| \phi_{n,x} \|_{L^p} = 2^{dn \Ll( \frac12 - \frac1p \Rr)} \| \phi_{0,x} \|_{L^p}.
\end{equation}
For $p = +\infty$ we estimate $\|\V_n f\|_{L^\infty}$ using
\begin{equation*}
| \V_n f (y) | \lesssim R^d \sup_{x \in \Lambda_n} | v\n f| |\phi_{n,x}(y)| \lesssim \|\phi_{n,0}(y)\|_{L^{\infty}} \sup_{x \in \Lambda_n} | v\n f|  .
\end{equation*}
This yields the upper bound for $\|\V_n f\|_{L^p}$.

\smallskip

As for the rightmost inequality, notice that $v_{n,x}(\V_n f)=v_{n,x}f$, that is, $v_{n,x}f=\int \phi_{n,x}(y) \V_nf(y) \d y$. Let $p < +\infty$ and $p'$ be its conjugate exponent. By Hölder's inequality, 
\begin{equation*}
|v_{n,x}f| \leqslant \| \phi_{n,x} \|_{L^{p'}} \| \V_nf  \mathbbm{1}_{B(x,2^{-n}R)} \|_{L^{p}} ,
\end{equation*}
and moreover,
\begin{equation*}
\sum_{x \in \Lambda_n} \int | \V_n f (y) |^p \mathbbm{1}_{B(x,2^{-n}R)} (y) \d y = \int | \V_n f (y) |^p \sum_{x \in \Lambda_n} \mathbbm{1}_{B(x,2^{-n}R)} (y) \d y \lesssim \| \V_n f \|_{L^p}^p .
\end{equation*}
By \eqref{e:phiscaling}, we have $\| \phi_{n,x} \|_{L^{p'}} \lesssim 2^{dn (\frac{1}{2}-\frac{1}{p'})} = 2^{- dn \Ll( \frac12 - \frac1p \Rr)}$, and this concludes the proof for the case $p<+\infty$. For $p = +\infty$, we just notice that $
|v_{n,x}f| \leqslant \|\V_n f\|_{L^\infty} \| \phi_{n,x} \|_{L^1}
$. 
\end{proof}

\begin{rem}
For each $k \in \Z$, the norm
\begin{equation*}
\| f \|_{\B^{\al,k}_{p,q}} = \Ll\| \Ll( v_{k,x} f\Rr)_{x \in \Lambda_k} \Rr\|_{\ell^p} + \Ll\| \left( 2^{\alpha n} 2^{dn \Ll( \frac12 - \frac1p \Rr) } \Ll\| \Ll( w\ni f\Rr)_{i < 2^d, x \in \Ln} \Rr\|_{\ell^p} \right)_{n \geqslant k} \Rr\|_{\ell^q}
\end{equation*}
is equivalent to that in \eqref{e:def:Besov}. This is easy to show using Proposition~\ref{p:charact-coefs} and the definition of multiresolution analysis.
\end{rem}


As we now show, for $\al < 0$, the Besov space $\B^{\al}_{\infty,\infty}$ of Definition~\ref{def:Besov:wave} coincides with the Besov-Hölder space $\C^{\alpha}$ given by Definition~\ref{def.Besov}.

\begin{prop}
\label{p.equiv}
Let $\alpha < 0$. There exist $C_1,C_2 \in (0,\infty)$ such that for every $f \in C^\infty_c$, we have
\begin{equation}
\label{e:comp1}
C_1 \|f\|_{\C^\al} \le \|f\|_{\B^{\al}_{\infty,\infty}} \le C_2 \|f\|_{\C^\al}.
\end{equation}
\end{prop}
\begin{proof}
The result is classical and proved e.g.\ in \cite[Proposition~3.20]{struc}. We recall the proof for the reader's convenience. One can check that there exists $C < \infty$ such that for every $f \in C^\infty_c$, $n \in \Z$ and $x \in \Rd$,
\begin{equation}  \label{e:comp11}
2^{\frac {dn}{2}} |w\ni f| \le C \|f\|_{\C^\al},
\end{equation}
and this yields the second inequality in \eqref{e:comp1}. Conversely, we let $f \in C^\infty$ satisfy $\|f\|_{\B^{\al}_{\infty,\infty}} \le 1$. We aim to show that there exists a constant $C < \infty$ (independent of~$f$) such that for every $y \in \Rd$, $\eta \in \mathscr{B}^{r_0}$ (with $r_0= - \lfloor \alpha \rfloor$) and $\lambda \in (0,1]$, we have
\begin{equation*}  
\lambda^{-\al} \int_\Rd f \, \lambda^{-d} \eta \Ll( \frac{\cdot - y}{\lambda} \Rr)  \le C.
\end{equation*}
We write $\eta_{\lambda,y} := \lambda^{-d} \eta((\cdot - y)/\lambda)$, and observe that
\begin{equation*}  
\int f \, \eta_{\lambda,y} = \sum_{x \in \Lambda_0} (v_{0,x} f)(v_{0,x} \eta_{\lambda,y}) + \sum_{i < 2^d} \sum_{n \ge 0}\sum_{x \in \Lambda_n} (w\ni f)(w\ni \eta_{\lambda,y}).
\end{equation*}
We consider only the second term of the sum above, as the first one can be obtained with the same technique. By the definition of $\|f\|_{\B^{\al}_{\infty,\infty}}$, for every $n \ge 0$, we have
\begin{equation}
\label{e:obv}
2^{\frac {dn} 2} |w\ni f| \le C 2^{-\al n}.
\end{equation}
In order for $w\ni \eta_{\lambda,y}$ to be non-zero, we must have $|x-y| \le C (\lambda \vee 2^{-n})$. Moreover, by a Taylor expansion of $\eta$ around $x$ and \eqref{e:vanish.moment}, we have
\begin{equation}
\label{e:taylor1}
2^{-n} \le \lambda \quad \implies \quad 2^{\frac {dn} 2}  |w\ni \eta_{\lambda,y} | \le C 2^{-r n}  \, \lambda^{-d-r},
\end{equation}
while 
\begin{equation}
\label{e:taylor11}
2^{-n} \ge \lambda \quad \implies \quad 2^{\frac {dn} 2} |w\ni \eta_{\lambda,y} | \le C  2^{dn}.
\end{equation}
and the same bound holds for $2^\frac{dn}{2}| v_{0,x} \eta_{\lambda,y} |$. 
For each $n \geqslant 0$, there exists a compact set $K_n \subset \Lambda_n$ independent from $f$ such that the condition $w\ni \eta_{\lambda,y} \neq 0$ implies that $x \in K_n$. Since the sum over $x \in \Lambda_n \cap K_n$ has less than $C 2^{nd}$ terms, the result follows.
\end{proof}

\begin{rem}
Notice that we can replace $r_0 = - \lfloor \alpha \rfloor$ by a generic integer $r > | \alpha |$ in Definition \ref{def.Besov}, obtaining an equivalent norm. Indeed, Proposition~\ref{p.equiv} shows that it suffices to control the behavior of $f$ against shifted and rescaled versions of the wavelet functions $\phi$ and $\psii$.
\end{rem}

\begin{rem}
\label{r.really.holder}
In view of Proposition~\ref{p.equiv}, when $\alpha < 0$, we have $\C^\al = \B^\al_{\infty,\infty}$, and $\Cl(U) = \Bl{\al}_{\infty,\infty}(U)$ where $\C^\al $ and $\Cl(U)$ are given by Definition~\ref{def.Besov}. By extension, we set 
$$
\C^\al := \B^\al_{\infty,\infty} \quad \text{ and } \quad \Cl(U) := \Bl{\al}_{\infty,\infty}(U)
$$ 
for every $\al \in \R$. Although we will not use this fact here, note that for $\al \in (0,1)$, there exists a constant $C < \infty$ such that for every $f \in C^\infty_c$,
\begin{equation}  
\label{e.really.holder}
C^{-1} \, \|f\|_{\C^\al} \le \|f\|_{L^\infty} + \sup_{0 < |x-y| \le 1} \frac{|f(y) - f(x)|}{|y-x|^\al} \le C \, \|f\|_{\C^\al}.
\end{equation}
The proof of this fact can be obtained  similarly to that of Proposition~\ref{p.equiv} (see also \cite[Theorem~6.4.5]{mey}). Hence, one can show that for $\al \in (0,\infty) \setminus \N$, the space $\C^\al$ is the separable version of the space $C^{\lfloor \alpha \rfloor}$ of functions whose derivative of order $\lfloor \alpha \rfloor$ is $(\al - \lfloor \alpha \rfloor)$-H\"older continuous. (By ``separable version of'', we mean that there is a natural norm associated with the space just described, and we take the completion of the space of smooth functions with respect to this norm.) For $\alpha \in \N$, the space $\C^\al$ is stricty larger than the (separable version of) the space of $C^\al$ functions. We refer to \cite{bcd} for details. 
\end{rem}

The following proposition is a weak manifestation of the multiplicative structure of Besov spaces, which is exposed in more details in the appendix. 
\begin{prop}[multiplication by a smooth function]
\label{p:mult-2}
Let $r > | \alpha | $ and $p,q \in [1,\infty]$. For every $\chi \in C^{r}_c$, the mapping $f \mapsto \chi f$ extends to a continuous functional from $\Bb$ to itself.
\end{prop}

\begin{proof}[Partial proof of Proposition~\ref{p:mult-2}]
We give a proof for the particular case $\al < 0$ and $p = q = \infty$. The general case is postponed to the appendix. 
Let $f \in C^\infty_c$ and consider the integral
\begin{equation*}
\lambda^{-d}  \int f (y) \chi (y)   \eta \left(\frac{y-x}{\lambda}\right) \d y .
\end{equation*}
For every $\lambda > 0$ and $x \in \R^d$, define $\tilde{\eta}$ as: $\tilde{\eta}_{\lambda,x}\left(\frac{y-x}{\lambda}\right) = \chi (y) \eta \left(\frac{y-x}{\lambda}\right)$. Then $\tilde{\eta}_{\lambda,x}(z) = \chi (z \lambda + x) \eta(z)$ for $z \in \R^d$. One can notice that $\tilde{\eta}_{\lambda ,x} \in C^{r_0}_c$ and $\supp \tilde{\eta}_{\lambda ,x} \subset \supp \eta$. Hence, by Proposition~\ref{p.equiv}, there exists $C>0$ (possibly different in every line) such that: 
\begin{align*}
\lambda^{-d}  \int f (y) \chi (y)   \eta \left(\frac{y-x}{\lambda}\right) \d y & \leqslant C \lambda^\alpha \|f\|_{\B^\al_{\infty,\infty}} \|\tilde{\eta}_{\lambda,x} \|_{C^{r_0}_c} \\
& \leqslant C \lambda^\alpha \|f\|_{\B^\al_{\infty,\infty}} \| \chi(\lambda \, \cdot\, ) \|_{C^{r_0}_c} \\
& \leqslant C \lambda^\alpha \|f\|_{\B^\al_{\infty,\infty}} \| \chi \|_{C^{r_0}_c},
\end{align*}
uniformly over $f \in C^\infty_c$, $\lambda \in (0,1]$, $\eta \in \mathscr{B}^{r_0}$ and $x \in \R^d$. The result follows by the fact that $C^\infty_c$ is dense in $\B^{\alpha}_{\infty,\infty}$.
%
\end{proof}

\begin{rem}
\label{r:topo-loc}
The notion of a complete space makes sense for arbitrary topological vector spaces, since a description of neighbourhoods of the origin is sufficient for defining what a Cauchy sequence is. Yet, in our present setting, the topology of $\Bbl(U)$ is in fact metrisable. To see this, note that there is no loss of generality in restricting the range of $\chi$ indexing the semi-norms to a countable subset of $C^\infty_c(U)$, e.g.\ $\{\chi_n, n \in \N\}$ such that for every compact $K \subset U$, there exists $n$ such that $\chi_n = 1$ on $K$. Indeed, it is then immediate from Proposition~\ref{p:mult-2} that if $\chi$ has support in $K$, then $\|\chi f\|_{\Bb} \le C \|\chi_n f\|_{\Bb}$ for some $C$ not depending on $f$. Hence, we can view $\Bbl(U)$ as a complete (Fréchet) space equipped with the metric
\begin{equation}
\label{e:def:metric}
d_{\Bbl(U)}(f,g) = \sum_{n = 0}^{+\infty} 2^{-n} \ \| \chi_n (f-g) \|_{\Bb} \wedge 1 .
\end{equation}
\end{rem}

We now give an alternative family of semi-norms, based on wavelet coefficients, that is equivalent to the family given in Definition \ref{def:Besov:wave} or Remark~\ref{r:topo-loc}.

\begin{defi}[spanning sequence]
\label{def:span}
Recall that $R$ is such that \eqref{e:def:R} holds. Let $K \subset U$ be compact and $k \in \N$. We say that the pair $(K,k)$ is \emph{adapted} if
	\begin{equation}
		\label{e:def:n0}
		2^{-k} R < \mathsf{dist}(K,U^c).
	\end{equation}
We say that the set $\msc K$ is a \emph{spanning sequence} if it can be written as 
$$
\msc K = \{(K_n,k_n),\ n \in \N\},
$$
where $(K_n)$ is an increasing sequence of compact subsets of $U$ such that $\bigcup_n K_n = U$ and for every $n$, the pair $(K_n,k_n)$ is adapted.
\end{defi}

For every adapted pair $(K,k)$, $f \in C^\infty_c(U)$ and $n \ge k$, we let
	\begin{equation}
	\label{e:def:vnk}
	v\nkold f  =  2^{dn \Ll( \frac12 - \frac1p \Rr) } \Ll\| \Ll( v\n f\Rr)_{x \in \Ln \cap K} \Rr\|_{\ell^p},
	\end{equation}
	\begin{equation}
	\label{e:def:wnk}
	w\nkold f  =  2^{dn \Ll( \frac12 - \frac1p \Rr) } \Ll\| \Ll( w\ni f\Rr)_{i < 2^d, x \in \Ln \cap K} \Rr\|_{\ell^p},
	\end{equation}
	and we define the semi-norm
	\begin{equation}
	\label{e:def:local-norm}
	\|f\|_{\B^{\al,K,k}_{p,q}} = v\nzkold f + \Ll\| \Ll( 2^{\al n} w\nkold f \Rr)_{n \ge k}  \Rr\|_{\ell^q}. 
	\end{equation}
	
\begin{prop}[Local Besov spaces via wavelet coefficients]
\label{p:loc-wave}
Let $p,q \in [1,\infty]$. 

\noindent (1) For every adapted pair $(K,k)$, the mapping $f \mapsto \|f \|_{\B^{\al,K,k}_{p,q}}$ extends to a continuous semi-norm on $\Bbl(U)$. 

\noindent (2) The topology induced by the family of semi-norms $\| \, \cdot \, \|_{\B^{\al,K,k}_{p,q}}$, indexed by adapted pairs $(K,k)$, is that of $\Bbl(U)$. 

\noindent (3) Let $\msc K$ be a spanning sequence. Part (2) above remains true when considering only the seminorms indexed by pairs in $\msc K$.
\end{prop}
\begin{rem}
\label{r:2nd-metric}
Another metric that is compatible with the topology on $\Bbl(U)$ is thus given by
$$
d'_{\Bbl(U)}(f,g) = \sum_{n = 0}^{+\infty} 2^{-n}\ \| f-g \|_{\B^{\al,K_n,k_n}_{p,q}} \wedge 1 ,
$$
where $\msc K = \{(K_n,k_n), n \in \N\}$ is any given spanning sequence.
\end{rem}
\begin{proof}[Proof of Proposition~\ref{p:loc-wave}]
%
In order to prove parts (1-2) of the proposition, it suffices to show the following two statements.
\begin{equation}
\label{e:impl}
\begin{array}{c}
\text{For every adapted pair } (K,k), \text{ there exists } \chi \in C^\infty_c(U) \text{ and } C < \infty \text{ s.t.} \\
\forall f \in C^\infty(U), \ \|f \|_{\B^{\al,K,k}_{p,q}} \le C \|\chi f\|_{\Bb};
\end{array}
\end{equation}
\begin{equation}
\label{e:revers}
\begin{array}{c}
\text{For every } \chi \in C^\infty_c(U), \text{ there exists } (K,k) \text{ adapted pair and } C < \infty \text{ s.t.} \\
\forall f \in C^\infty(U), \ \|\chi f\|_{\Bb} \le C \|f \|_{\B^{\al,K,k}_{p,q}}.
\end{array}
\end{equation}
We begin with \eqref{e:impl}. Let $(K,k)$ be an adapted pair, and let $\chi \in C^\infty_c(U)$ be such that $\chi = 1$ on $K + \ov B(2^{-k}R)$. For every $n \ge k$ and $x \in \Ln \cap K$,
$$
v\n f  = v\n(\chi f), \qquad w\ni f = w\ni(\chi f) \quad (i < 2^d),
$$
and as a consequence,
$$
v\nkold(f) \le 2^{dn \Ll( \frac12 - \frac1p \Rr) } \Ll\| \Ll( |v\n(\chi f)|\Rr)_{x \in \Ln} \Rr\|_{\ell^p} \le C \|\V_n (\chi f)\|_{L^p}
$$
(where we used \eqref{e:comp_coeff} in the last step), and similarly with $v\nkold, v\n$ and $\V_n$ replaced by $w\nkold, w\ni$ and $\W_n$ respectively. We thus get that
\begin{align*}
\|f\|_{\B^{\al,K,k}_{p,q}} & = v\nzkold f + \Ll\| \Ll( 2^{\al n} w\nkold f \Rr)_{n \ge k}  \Rr\|_{\ell^q} \\ 
 & \le C \Ll( \|\V_{k} (\chi f)\|_{L^p} + \Ll\| \Ll( 2^{\al n} \|\W_n (\chi f)\|_{L^p}  \Rr)_{n \ge n_0}  \Rr\|_{\ell^q} \Rr) 
 \le C \|\chi f\|_\Bb.
\end{align*}
We now turn to \eqref{e:revers}. In order to also justify part (3), we will show that we can in fact pick the adapted pair in $\msc K = \{(K_n,k_n), n \in \N\}$.

Let $(K,k)$ be an adapted pair. For every $f \in C^\infty(U)$, we define
\begin{equation}
\label{e:def:fK}
f_{K} = \sum_{x \in \Lambda_{k} \cap K} v_{k,x}(f) \, \phi_{k,x} + \sum_{\substack{n \ge k, i < 2^d\\ x \in \Ln \cap K}} w\ni(f) \, \psi\ni.
\end{equation}
The functions $f$ and $f_K$ coincide on
\begin{equation}
\label{e:def:K'}
K':= \Ll\{x \in \R^d : \msf{dist}(x,K^c) \ge  2^{-k} R \Rr\}.
\end{equation}
(Although the notation is not explicit in this respect, we warn the reader that $f_K$ and $K'$ are defined in terms of the pair $(K,k)$ rather than in terms of $K$ only.)
Let $\chi \in C^\infty_c(U)$ with compact support $L \subset U$. Assuming that
\begin{equation}
\label{e:existn}
\text{there exists } n \in \N \text{ s.t. } L \subset K_n',
\end{equation}
we see that for such an $n$,
$$
\|\chi f \|_\Bb = \|\chi f_{K_n} \|_\Bb \le C \|f_{K_n}\|_\Bb \le C \|f\|_{\B^{\al,{K_n},k_n}_{p,q}}
$$
by Proposition~\ref{p:mult-2} and \eqref{e:def:local-norm}. Hence, it suffices to justify \eqref{e:existn}.
Let $\msf{d} = \msf{dist}(L,U^c)$. Since $x \mapsto \msf{dist}(x,U^c)$ is positive and continuous on $L$, we obtain $\msf{d}>0$. If $U$ is bounded, then there exists $n \in \N$ such that $K_n$ contains the compact set $\{ x : \msf{dist}(x,U^c) \ge \msf{d}/2\}$. We must then have $2^{-k_n} R < \msf{d}/2$, so that
\begin{align*}
	x \in L & \Rightarrow \msf{dist}(x,K_n^c) \ge \msf{dist}(x,U^c) - \frac{\msf{d}}{2} \ge \frac{\msf{d}}{2} > 2^{-k_n} R \\
	& \Rightarrow x \in K_n'.
\end{align*}
If $U$ is unbounded, we can do the same reasoning with $U$ replaced by 
$$
U \cap \Ll( L + B(0,R) \Rr),
$$
so the proof is complete.
\end{proof}

\begin{rem}
For any adapted pair $(K,k)$, the quantity $\|f\|_{\B^{\al,K,k}_{p,q}}$ is well defined as an element of $[0,+\infty]$ as soon as $f$ is a linear form on $C^r_c(U)$, through the interpretation of $v_{k,x} f$ and $w\ni f$ in \eqref{e:def:vw} as a duality pairing. 
\end{rem}

The characterization of Proposition~\ref{p:loc-wave} yields a straightforward proof of embedding properties between Besov spaces (see for example \cite[Proposition 2.71]{bcd}).
\begin{prop}[Local Besov embedding]
	\label{p:loc-emb}
Let $1 \leqslant p_2 \leqslant p_1 \leqslant +\infty$, $1 \leqslant q_2 \leqslant q_1 \leqslant +\infty$, $\alpha \in \R$ and
$$
\be = \alpha + d\Ll(\frac{1}{p_2} - \frac{1}{p_1}\Rr).
$$
If $|\al|, |\be| < r$ and $(K,k)$ is an adapted pair, then there exists $C < \infty$ such that for every linear form $f$ on $C^r_c(U)$,
$$
\|f\|_{\B^{\al,K,k}_{p_1,q_1}} \le C \|f\|_{\B^{\be,K,k}_{p_2,q_2}} .
$$
In particular, we have $\Bl{\be}_{p_2,q_2}(U) \subset \Bl{\al}_{p_1,q_1}(U)$.
\end{prop}
\begin{proof}
We write the norm \eqref{e:def:local-norm}, recall \eqref{e:def:vnk}-\eqref{e:def:wnk}, and use the fact that $\| \cdot \|_{\ell^{p_1}} \leqslant \| \cdot \|_{\ell^{p_2}}$ if $p_1 \geqslant p_2$.
\end{proof}
Due to our definition of the space $\Bbl(U)$ as a completion of $C^\infty(U)$, the fact that $\|f\|_{\B^{\al,K,k}_{p,q}}$ is finite for every adapted pair $(K,k)$ does not necessarily imply that $f \in \Bbl(U)$. We have nonetheless the following result.

\begin{prop}[A criterion for belonging to $\Bl{\al}_{p,q}(U)$]
\label{p:crit}
Let $|\al'| < r$ and let $p,q \in [1,\infty]$. Let $f$ be a linear form on $C^r_c(U)$, and let $\msc K$ be a spanning sequence. If for every $(K,k) \in \msc K$, 
$$
\|f\|_{\B^{\al',K,k}_{p,q}} < \infty,
$$
then for every $\al < \al'$, the form $f$ belongs to $\Bl{\al}_{p,1}(U)$.
\end{prop}


\begin{proof}[Proof of Proposition~\ref{p:crit}]
We first check that for every $(K,k) \in \msc K$, there exists a sequence $(f_{N,k})_{N \in \N}$ in $C^r_c(U)$ such that $\|f - f_{N,k}\|_{\B^{\al,K,k}_{p,1}}$ tends to $0$ as $N$ tends to infinity. The functions
$$
f_{N,k} := \sum_{x \in \Lambda_{k} \cap K} v_{k,x} (f) \, \phi_{k,x} + \sum_{\substack{k \le n \le N, i < 2^d\\ x \in \Ln \cap K}}  w\ni (f) \ \psi\ni 
$$
satisfy this property. Now notice that for $(\tilde{k},\tilde{K}) \in \msc K$ such that $\tilde{K} \supset K$, the function $f_{N,\tilde{k}}$ coincides with $f_{N,k}$ on the set $K'$ of \eqref{e:def:K'}. Then defining $f_N= f_{N,N}$,  we obtain that for every $\chi \in C^\infty_c (U)$, there exists $n_0,N_0(n_0)$ such that for every $n \ge n_0$ and $N \ge N_0$,
\begin{equation*}
\|(f_N - f ) \chi \|_{\B^{\al}_{p,1}} = \|(f_{N,k_n} - f ) \chi \|_{\B^{\al}_{p,1}},
\end{equation*}
where we have indexed the spanning sequence as $\msc K = (k_n,K_n)_{n \in \N}$.
By \eqref{e:revers}, there exist $(k_m,K_m)\in \msc K$, $C>0$ with $m$ large enough, such that:
\begin{equation*}
\|(f_{N,k_n} - f ) \chi \|_{\B^{\al}_{p,1}} \leqslant C \| f_{N,k_n} - f  \|_{\B^{\al,K_m,k_m}_{p,1}}
\end{equation*}
We can eventually choose $m=n$ to obtain $\|(f_N - f ) \chi \|_{\B^{\al}_{p,1}} \rightarrow 0$ for every$\chi \in C^\infty_c(U)$, which by Proposition \ref{p:loc-wave} is the needed result.
\end{proof}

Naturally, tightness criteria rely on the identification of compact subsets of the space of interest.

\begin{prop}[Compact embedding]
	\label{p:compact}
	Let $U$ be an open subset of $\R^d$. For every $\al < \al'$ and $p,q, s \in [1,+\infty]$, the embedding $\Bl{\al'}_{p,q}(U) \subset \Bl{\al}_{p,s}(U)$ is compact.

\end{prop}
\begin{proof} 

By Proposition~\ref{p:loc-wave} and the definition of boundedness in Fréchet spaces, a sequence $(f_m)_{m \in \N}$ of elements of $\Bl{\al'}_{p,q}(U)$ is bounded in $\Bl{\al'}_{p,q}(U)$ if and only if for every adapted pair $(K,k)$, we have
$$
\sup_{m \in \N} \|f_m\|_{\B^{\al',K,k}_{p,q}} < \infty.
$$
We show that for every adapted pair $(K,k)$, there exists a subsequence $(m_{n_k})_{n_k \in \N}$ and $f^{(K)}$ in $\Bl{\al}_{p,s}(U)$ such that $\|f_{m_{n_k}} - f^{(K)} \|_{\B^{\al',K,k}_{p,s}}$ tends to $0$ as $n$ tends to infinity. 
The assumption that $\sup_m  \|f_{m}\|_{\B^{\al',K,k}_{p,q}} < \infty$ can be rewritten as
$$
\Ll\| \Ll( v_{k,x} f_m\Rr)_{x \in \Lambda_k \cap K} \Rr\|_{\ell^p} + \\
\Ll\| \Ll( 2^{n\Ll[\al' + d \Ll( \frac12 - \frac1p \Rr)\Rr] } \Ll\| \Ll( w\ni f_m\Rr)_{i < 2^d, x \in \Ln \cap K} \Rr\|_{\ell^p} \Rr)_{n \ge k}  \Rr\|_{\ell^q} \le C,
$$
uniformly over $m \in \N$. By a diagonal extraction argument, there exist a subsequence, which we still denote $(f_m)$ for convenience, and numbers $\td v_{k,x}$, $\td w\ni$ such that 
\begin{multline*}
\Ll\| \Ll( v_{k,x} f_m - \td v_{k,x}\Rr)_{x \in \Lambda_k \cap K} \Rr\|_{\ell^p} + \\
\Ll\| \Ll( 2^{n\Ll[\al + d \Ll( \frac12 - \frac1p \Rr)\Rr] } \Ll\| \Ll( w\ni f_m - \td w\ni\Rr)_{i < 2^d, x \in \Ln \cap K} \Rr\|_{\ell^p} \Rr)_{n \ge k}  \Rr\|_{\ell^{s}} \xrightarrow[m \to \infty]{} 0.
\end{multline*}
Defining
$$
f^{(K)} = \sum_{x \in \Lambda_{k} \cap K} \td v_{k,x} \ \phi_{k,x} + \sum_{\substack{n \ge k, i < 2^d\\ x \in \Ln \cap K}} \td w\ni\ \psi\ni ,
$$
we have $f^{(K)} \in \Bl{\al}_{p,s}(U)$ and $\|f_{m} - f^{(K)}\|_{\B^{\al,K,k}_{p,s}} \to 0$ as $m$ tends to infinity. The subsequence $(f_m)$ is Cauchy in $\Bl{\al}_{p,s}(U)$. Indeed, for every $(K,k) \in \msc K$, there exists $n_0(K)$ such that for every $n,m \ge n_0$,
\begin{equation*}
\|f_n - f_m \|_{\B^{\al,K,k}_{p,s}} \leqslant \|f_n - f^{(K)} \|_{\B^{\al,K,k}_{p,s}} + \|f^{(K)} - f_m \|_{\B^{\al,K,k}_{p,s}} < \varepsilon.
\end{equation*} 
This completes the proof.
\end{proof}

\begin{rem}
	\label{r:loc-not-glob} 
	Proposition~\ref{p:compact} would not be true if $\Bl{\al'}_{p,q}(U)$ and $\Bl{\al}_{p,s}(U)$ were replaced by their global counterparts, respectively $\B^{\al'}_{p,q}$ and $\B^{\al}_{p,s}$. Indeed, one can take for example a non-zero function $f \in C^\infty_c$ and consider the sequence $\left( f(\cdot - np)\right)_{n \geqslant 1}$, with $p \in \R^d \setminus \{0\}$. This sequence is bounded in every global Besov space $\B^\al_{p,q}$, but has no convergent subsequence in any of these spaces.
\end{rem}
An immediate consequence of Propositions~\ref{p:crit} and \ref{p:compact} is:
\begin{cor}
\label{c:precomp}
Let $|\al'| < r$, $p,q \in [1,\infty]$, let $\msc K$ be a spanning sequence, and for every $(K,k) \in \msc K$, let $M_K \in [0,\infty)$. For every $\al < \al'$, $s \in [1,\infty]$, the set 
\begin{equation}
\label{e:precomp}
\Ll\{f \text{ linear form on } C^r_c(U) \text{ such that } \forall (K,k) \in \msc K, \ \|f\|_{\B^{\al',K,k}_{p,q}} \le M_K \Rr\}
\end{equation}
is compact in $\Bl{\al}_{p,s}(U)$.
\end{cor}

\begin{thm}[Tightness criterion] \label{t:tight}
	Recall that $\phi, (\psii)_{1 \le i < 2^d}$ are in $C^r_c$ and such that \eqref{e:def:R} holds, and fix $p \in [1,\infty)$, $q \in [1,\infty]$ and $\al, \be \in \R$ satisfying $|\al|,|\be| < r$, $\al < \be$. Let $(f_m)_{m \in \N}$ be a family of random linear forms on $C^r_c(U)$, and let $\msc K$ be a spanning sequence (see Definition~\ref{def:span}). Assume that for every $(K,k) \in \msc K$, there exists $C = C(K,k) < \infty$ such that for every $m \in \N$,
	\begin{equation}
		\label{e:tight1}
		\sup_{x \in \Lambda_{k} \cap K} \E \Ll[ \Ll| \langle f_m,\phi(2^{k}(\, \cdot \, - x)) \rangle \Rr|^p  \Rr]^{1/p} \le C,
	\end{equation}
	and
	\begin{equation}
		\label{e:tight2}
		\sup_{x \in \Lambda_{n} \cap K} 2^{dn} \ \E \Ll[ \Ll| \langle f_m,\psii(2^n(\, \cdot \, - x)) \rangle \Rr|^p  \Rr]^{1/p} \le C 2^{-n \be} \qquad (i < 2^d, n \ge k).
	\end{equation}
Then the family $(f_m)$ is tight in $\Bl{\al}_{p,q}$. If moreover $\al < \be - \frac d p$, then the family is also tight in $\Cl(U)$.
\end{thm}
\begin{proof}

By \eqref{e:def:scaling} and \eqref{e:def:vw}, we have for every $(K,k) \in \msc K$, uniformly over $m$ that
$$
\sup_{x \in \Lambda_{k} \cap K} \E\Ll[ \Ll|v_{k,x} f_m\Rr|^p \Rr] \lesssim 1,
$$
$$
\sup_{x \in \Ln \cap K} 2^{\frac{dnp}{2}} \ \E \Ll[ \Ll| w\ni f_m \Rr|^p  \Rr] \lesssim 2^{-np\be} \qquad (i < 2^d, n \ge k).
$$
Recalling the definition of $v\nzkold$ and $w\nkold$ in \eqref{e:def:vnk} and \eqref{e:def:wnk} respectively, we have
$$
\Ll|v\nzkold f_m\Rr|^p \lesssim \sum_{x \in \Lambda_{k} \cap K} \Ll|v_{k,x} f_m\Rr|^p,
$$
so that 
$$
\E\Ll[\Ll|v\nzkold f_m\Rr|^p \Rr] \lesssim 1.
$$
Similarly,
$$
\Ll|w\nkold f_m\Rr|^p \lesssim 2^{dn \Ll( \frac{p}{2} - 1 \Rr) }\sum_{i < 2^d, x \in \Lambda_{k} \cap K} \Ll|w\ni f_m\Rr|^p,
$$
so that
$$
\E\Ll[\Ll|w\nkold f_m\Rr|^p \Rr] \lesssim 2^{-np\be}.
$$
It follows from these two observations and from \eqref{e:def:local-norm} that
\begin{equation} \label{e:sequence-bound}
\sup_{m \in \N} \E\Ll[\|f_m\|^p_{\B^{\be,K,k}_{p,\infty}} \Rr] < \infty.
\end{equation}
By Chebyshev's inequality, for any given $\eps > 0$, there exist $(M_K)$ such that if we set
$$
\msc E := \Ll\{f \text{ linear form on } C^r_c(U) \text{ such that } \forall (K,k) \in \msc K, \ \|f\|_{\B^{\be,K,k}_{p,\infty}} \le M_K \Rr\},
$$
then for every $m$,
$$
\P[f_m \in \msc E] \ge 1-\eps.
$$
By Corollary~\ref{c:precomp}, this implies the tightness result in $\Bl{\al}_{p,q}(U)$. For the second statement, we note that \eqref{e:sequence-bound} and Proposition~\ref{p:loc-emb} imply that 
\begin{equation*}
\sup_{m \in \N}\E\Ll[\|f_m\|^p_{\B^{\be-d/p,K,k}_{\infty,\infty}} \Rr] < \infty .
\end{equation*}
The conclusion then follows in the same way.
\end{proof}
\begin{rem}  
We can also infer from the proof that for each $\chi \in C^\infty_c(U)$, there exists a constant $\widetilde C_\chi$ such that under the assumption of Theorem~\ref{t:tight}, we have
\begin{equation*}  
\sup_{m \in \N}  \E \Ll[ \|\chi f_m\|_{\B^{\be}_{p,\infty}}^p \Rr] < \widetilde C_\chi C,
\end{equation*}
as well as
\begin{equation*}  
\sup_{m \in \N}  \E \Ll[ \|\chi f_m\|_{\C^{\be - \frac{d}{p}}}^p \Rr] < \widetilde C_\chi C.
\end{equation*}
\end{rem}



%

We conclude this section by proving a statement analogous to the Kolmogorov continuity theorem. Recalling from Remark~\ref{r.really.holder} the interpretation of the space~$\C^\al$ as a H\"older space, the satement below can indeed be seen as a generalization of the classical result of Kolmogorov. (The fact that the statement can apply to positive exponents of regularity is due to the cancellation property \eqref{e:vanish.moment}.)
\begin{prop} \label{p:kolmotype} 
Let $(f(\eta), \eta \in C^r_c(U))$ be a family of random variables such that, for every $\eta, \eta' \in C^r_c(U)$ and every $\mu \in \R$, there exists a measurable set $A=A(\mu,\eta, \eta')$ with $\P(A)=1$ such that
\begin{equation}  
\label{e.linearity}
f(\mu \eta + \eta') (\omega)= \mu f(\eta) (\omega) + f(\eta') (\omega) \qquad \forall \omega \in A.
\end{equation}
Assume also the following weak continuity property: for each compact $K' \subset U$ and each sequence $\eta_n, \eta \in C^r_c(U)$ with $\supp \eta_n \subset K'$, we have
\begin{equation*}  
\eta_n \xrightarrow[n \to \infty]{\text{in } C^{r-1}_c} \eta \quad \implies \quad f(\eta_n) \xrightarrow[n \to \infty]{\text{prob.}} f(\eta).
\end{equation*}
Let $p \in [1,\infty)$, $q \in [1,\infty]$, and let $\al, \be \in \R$ be such that $|\al|, |\be| < r$ and $\al < \be$. Let $\msc K$ be a spanning sequence, and assume finally that, for every $(K, k) \in \msc {K}$, there exists $ C > 0$ such that for every $n \geqslant k$, 
  \begin{equation*} \sup_{x \in \Lambda_k \cap K} \mathbb{E} \left[ \left| f\Ll(\phi (2^k
     (\cdot - x)) \right) \right|^p \right]^{\frac{1}{p}} \leqslant C 
     \end{equation*}
  and
  \begin{equation*} \sup_{x \in \Lambda_n \cap K} 2^{dn} \mathbb{E} \left[ \left| f\left(  \psi (2^n
     (\cdot - x)) \right) \right|^p \right]^{\frac{1}{p}} \leqslant C 2^{- n \beta}. \end{equation*}
     Then there exists a random distribution $\td f$ taking values in $\Bl{\al}_{p,q}(U)$ such that for every $\eta \in C^r_c(U)$,
     \begin{equation}  
     \label{e.identity.tdf.f}
     \Ll( \td f, \eta \Rr) = f(\eta) \qquad \text{a.s.}
     \end{equation}
     Moreover, if $\al < \be - \frac d p$, then $\td f$ takes values in~$\Cl(U)$ with probability one.
\end{prop}

\begin{proof}
For every $(K,k) \in \msc K$ and $N \in \N$, we define
$$
\tilde{f}_{N,k} := \sum_{x \in \Lambda_{k} \cap K} v_{k,x} (f) \, \phi_{k,x} + \sum_{\substack{k \le n \le N, i < 2^d\\ x \in \Ln \cap K}}  w\ni (f) \ \psi\ni ,
$$
where we set
\begin{equation*}  
v_{k,x}(f) := f(\phi_{k,x}) \quad \text{ and } \quad w\ni(f) = f(\psi\ni).
\end{equation*}
Cleary, $\tilde{f}_{N,k}$ is almost surely in $C^r_c$. Following the proof of Theorem~\ref{t:tight}, we get:
$$
\E \Ll[ 2^{dn (\frac{p}{2}-1)} \sum_{x \in \Lambda_n \cap K, i < 2^d} |  w\ni (f)  |^p  \Rr] \lesssim 2^{-n p \beta} ,
$$
where the implicit constant does not depend on $n$. Hence, for each $\be' <  \be$ and each fixed integer $k$, we deduce by the Chebyshev inequality and the Borel-Cantelli lemma that $(\tilde f_{N,k})_N$ is a Cauchy sequence in $\B^{\be'}_{p,\infty}$ with probability one. We denote the limit by~$\tilde f_k$. It is clear that $\tilde f_k$ converges to some element $\tilde f$ of $\Bl{\be'}_{p,\infty}(U)$ as $k$ tends to infinity, since for each $\chi \in C^r_c$ with compact support in $U$, the sequence $\chi \td f_k$ is eventually constant as $k$ tends to infinity. By Proposition~\ref{p:loc-emb}, if $\al < \be - \frac d p$, then $\td f \in \Cl(U)$ with probability one. There remains to check that for every $\eta  \in C^r_c(U)$, the identity \eqref{e.identity.tdf.f} holds. By the orthogonality properties of $(\phi_{k,x}, \psi\ni)$ and the fact that $\eta$ has compact support in $U$, we have, for $k$ sufficiently large,
\begin{equation*}  
\eta = \sum_{x \in \Lambda_{k} \cap K} (\phi_{k,x}, \eta) \, \phi_{k,x} + \lim_{N \to + \infty} \sum_{\substack{k \le n \le N, i < 2^d\\ x \in \Ln \cap K}}  (\psi\ni,\eta) \ \psi\ni ,
\end{equation*}
where we recall that $(\cdot, \cdot)$ denotes the scalar product of $L^2(\R^d)$. We fix such $k$ sufficiently large, and denote
\begin{equation*}  
\eta_N := \sum_{x \in \Lambda_{k} \cap K} (\phi_{k,x}, \eta) \, \phi_{k,x} +  \sum_{\substack{k \le n \le N, i < 2^d\\ x \in \Ln \cap K}}  (\psi\ni,\eta) \ \psi\ni .
\end{equation*}
By a Taylor expansion of $\eta$ and \eqref{e:vanish.moment}, one can check that there exists $C(d,\eta) < \infty$ such that
\begin{equation*}  
2^{\frac{dn}{2}}\Ll|(\psi\ni,\eta)\Rr| \le C 2^{-rn}.
\end{equation*}
From this, together with the expressions for $\eta_N$ and $\eta$ above, we obtain that $\exists C(d,\eta) < \infty$ such that for any multi-index $\alpha \leqslant |r|$
$$
\| \partial^\alpha \eta - \partial^\alpha \eta_N \|_{L^\infty} < C \sum_{n > N} 2^{-r n} 2^{|\alpha | n}
$$ 
and thus
$$
\eta_N \xrightarrow[N \to \infty]{\text{in } C^{r-1}_c} \eta.
$$
Therefore by the weak continuity assumption, we deduce that
$$
f(\eta_N) \xrightarrow[N \to \infty]{\text{prob.}} f(\eta).
$$
In order to conclude, there remains to verify that
\begin{equation*}  
(\td f,\eta_N) = f(\eta_N) \qquad \text{a.s.}
\end{equation*}
This follows from the assumption \eqref{e.linearity}.
\end{proof}

\section{Application to the critical Ising model} \label{s:isingapp}

In this section, we apply the tightness criterion presented in Theorem~\ref{t:tight} to the magnetization field of the two-dimensional Ising model at the critical temperature. We will use extensively some basic notions related to the FK percolation model \cite{fk} and its relation to the Ising model via the Edwards-Sokal coupling \cite{es}. 

\subsection{Introduction to the random cluster model}
The random cluster model, or FK percolation model, was first introduced in \cite{fk}. We refer to \cite{grim} for a comprehensive book on the subject.


  
Let $\Lambda$ be a finite subset of $\Z^d$, $\Omega = \{0,1\}^{\E^d}$ with $\E^d$ the set of edges of the graph $\Z^d$, and $\mathcal{F}$ be the $\sigma$-algebra generated by cylinder sets. For $\omega \in \Omega$, let $\omega_e$ be the component of $\omega$ at $e \in \E^d$. Let $E_\Lambda = \{e=\langle x , y \rangle \in \E^d \mid x \in \Lambda, y \in \Lambda \}$ the set of edges with both endpoints in $\Lambda$. For $\xi \in \Omega$, define the following finite subset of $\Omega$:
\begin{equation*}
\Omega^\xi_\Lambda = \{ \omega \in \Omega \mid \omega_e = \xi_e  \quad \forall  e \in \E^d  \setminus E_\Lambda \}.
\end{equation*}
\begin{defi} \label{def:phiboundary}
Let $p \in [0,1]$, $q \in (0,\infty)$. The FK probability measure on $(\Omega,\mathcal{F})$ with boundary condition $\xi$ is
\begin{equation}
\phi^\xi_{\Lambda,p,q}(\omega) = \begin{cases} \frac{1}{Z_{\xi,\Lambda}} \left[ \prod_{e \in E_\Lambda} p^{\omega_e} (1-p)^{1-\omega_e} \right] q^{k(\omega)}  &\mbox{if } \omega \in \Omega^{\xi}_\Lambda \\ 
0 & \mbox{otherwise }  \end{cases} 
\end{equation} 
with $Z_{\xi,\Lambda} (p,q) = \sum_{\omega \in \Omega^\xi_\Lambda} \left[ \prod_{e \in E_\Lambda} p^{\omega_e} (1-p)^{1-\omega_e} \right] q^{k(\omega)}$ and $k(\omega)$ the number of connected components of the graph $(\Z^d,\eta(\omega))$, with $\eta(\omega) = \{e \in \E^d \mid \omega_e =1 \}$. 
\end{defi}
We will call the edge $e$ \emph{open} if $\omega_e =1$, and \emph{closed} otherwise. We call \emph{open clusters} the connected components of $(\Z^d,\eta(\omega))$, and write $x \leftrightarrow y$ if $x,y$ are in the same open cluster, $x \nleftrightarrow y$ otherwise. An \emph{open path} is a (possibly infinite) sequence $(e_i)$ of edges belonging to $\eta(\omega)$. The boundary condition is \emph{free} if $\xi_e =0$ $\forall e \in \E^d$ and \emph{wired} if $\xi_e =1$ $\forall e \in \E^d$.


\begin{rem}\label{r:indep}
For both free and wired boundary conditions, if the domain $\Lambda$ is the union of two subsets $\Lambda_1$ and $\Lambda_2$ such that $E_{\Lambda_1} \cap E_{\Lambda_2} = \emptyset$, then the configurations on $\Lambda_1$ and $\Lambda_2$ are independent. Indeed, calling $\overline{k(\omega , \E^d \setminus E_\Lambda)}$ the number of open clusters of $\omega$ that do not intersect $\E^d \setminus E_\Lambda$, we have $\overline{k(\omega , \E^d \setminus E_\Lambda)} = \overline{k(\omega , \E^d \setminus E_{\Lambda_1})} + \overline{k(\omega , \E^d \setminus E_{\Lambda_2})}$.
\end{rem}

Although in general the states on two different edges are not independent, the model exhibits a ``domain Markov'' \cite{d-chn} or ``nesting'' \cite{grim} property. Let $\mathcal{F}_{\Lambda}$ (respectively $\mathcal{T}_\Lambda$) be the $\sigma$-algebra generated by the states of edges in $E_\Lambda$ (respectively in $\E^d \setminus E_\Lambda$). We have the following result. 
\begin{lem}[{\cite[Lemma~4.13]{grim}}] \label{l:domainmarkov}
Let $p \in [0,1]$, $q \in (0,\infty)$, and let $\Lambda, \Delta$ be finite subsets of $\Z^d$ with $\Lambda \subset \Delta$. For every $\xi \in \Omega$, every event $A \in \mathcal{F}_{\Lambda}$ and every $\omega \in \Omega_\Delta^\xi$,
\begin{equation}
\phi_{\Delta,p,q}^\xi (A \mid \mathcal{T}_{\Lambda})(\omega) = \phi_{\Lambda,p,q}^\omega (A) .
\end{equation}
\end{lem}

The set $\Omega=\{ 0,1 \}^{\E^d}$ has a partial ordering given by $\omega \leqslant \omega'$ if $\forall e \in \E^d$ $\omega_e \leqslant \omega'_{e}$. A function $X : \Omega \rightarrow \R$ is called \emph{increasing} if $\omega \leqslant \omega' \Rightarrow X(\omega) \leqslant X(\omega')$. Likewise, an event $A \in \mathcal{F}$ is called increasing if the random variable $\mathbbm{1}_A$ is increasing. 
As a direct consequence of the FKG inequality and \cite[Lemma~4.14]{grim}, we have the following monotonicity properties.


\begin{lem} \label{l:fkg} 
Let $p \in [0,1]$, $q \geqslant 1$ and $\Lambda \subset \Delta \subset \Z^d$ finite sets. Then:
\begin{itemize}
\item For every $\eta \leqslant \xi \in \Omega$ and for every increasing event $A$:
\begin{equation*}
\phi^\eta_{\Lambda,p,q} (A) \leqslant \phi^\xi_{\Lambda,p,q} (A).
\end{equation*} 
\item For every increasing event $A \in \mathcal{F}_{\Lambda}$:
\begin{equation*}
\phi^1_{\Delta,p,q} (A) \leqslant \phi^1_{\Lambda,p,q} (A)
\end{equation*} 
\end{itemize}
\end{lem}


For $p \in [0,1]$, $q \geqslant 1$, the random cluster measure $\phi^\xi_{\Lambda,p,q}$ for both free and wired boundary conditions admits a thermodynamic limit as $\Lambda \rightarrow \Z^d$ \cite[Theorems~4.17 and~4.19]{grim}, which we call $\phi^\xi_{p,q}$. 
For every boundary condition $\xi$ such that $\phi^\xi_{\Lambda,p,q}$ admits a limit and every increasing event $A$, we have
\begin{equation*}
\phi^0_{p,q}(A) \leqslant \phi^\xi_{p,q}(A) \leqslant \phi^1_{p,q} (A).
\end{equation*}

\subsection{Relation with the 2-d Ising model} \label{s:escoupling}
Now consider the Ising-Potts model on a finite set $\Lambda \subset \Z^d$ as follows. Take a configuration space $\Sigma^0_{\Lambda} = \{ -1, 1 \}^{\Lambda}$. 
The Ising probability measure with \emph{free} boundary condition on $\Lambda$ is defined by
\begin{equation} \label{e:def:isingfree}
\pi_{G}^0(\sigma) = \frac{1}{Z_I} e^{ - \beta H(\sigma)}  \qquad H(\sigma)=- \sum_{e \in E_\Lambda} \mathbbm{1}_{\sigma_e = 1},
\end{equation} 
with $\beta > 0$, $\sigma_e = \sigma_x \sigma_y$  and  $Z^0_I  (\beta)= \sum_{\sigma \in \Sigma^0_\Lambda} e^{ - \beta H(\sigma)}$ .

Similarly, let $\Sigma^1_\Lambda = \{ \sigma \in \{ -1, 1 \}^{\Lambda} \mid \sigma_x = 1 \mbox{ } \forall x \in \partial \Lambda \}$. The Ising probability measure with $+$  boundary condition on $\Lambda$ is defined  as
\begin{equation} \label{e:def:isingplus}
\pi^{1}_{\Lambda}(\sigma) = \frac{1}{Z^1_I} e^{ - \beta H(\sigma)} \mathbbm{1}_{\Sigma^1_\Lambda}(\sigma) 
\end{equation}
 with $Z_I^1  (\beta)= \sum_{\sigma \in \Sigma^1_\Lambda} e^{ - \beta H(\sigma)}$ .
Random variables $\sigma_x$ for $x \in \Z^d$ are called spins.

\begin{rem}
Traditionally, the Hamiltonian of the Ising model is written as 
$$
H'(\sigma)=- \sum_{x \sim y} \sigma_x \sigma_y
$$
with $x \sim y$ nearest neighbours. Defining $\lambda_{\beta}(\sigma) \propto e^{ - \beta' H'(\sigma)}$ for the usual Ising measure, we recover it as $\lambda_{\beta/2} \sim \pi_{\beta}$.
\end{rem}

The Edwards-Sokal coupling on $\Lambda$ with boundary condition $\xi\in \{ 0,1 \}$ consists of defining the probability measure on $\Sigma^\xi_\Lambda \times \Omega$
\begin{equation} \label{e:def:escoupling}
\mu^\xi_\Lambda (\sigma, \omega) = \frac{1}{Z^\xi_{ES}} \prod_{e \in E^\xi_\Lambda} \left[ (1-p) \mathbbm{1}_{\omega_e = 0} + p \mathbbm{1}_{\omega_e = 1} \mathbbm{1}_{\sigma_e = 1}   \right]  \mathbbm{1}_{\Omega^\xi_\Lambda}(\omega)
\end{equation}
with $Z^\xi_{ES}$ such that $\sum_{(\sigma,\omega) \in \Sigma^\xi_\Lambda \times \Omega} \mu_\Lambda^\xi(\sigma , \omega ) =1$. 
From now on we fix
\begin{equation}
e^{-\beta} = 1-p \quad \mbox{ and } \quad q = 2.
\end{equation}

It is easy to obtain the following lemma (see \cite{grim}).
\begin{lem}\label{l:escoupling}
Let $p \in [0,1]$, $e^{-\beta} = (1-p)$, $q = 2$ and $\xi \in \{ 0,1 \}$. Let $\mu^\xi_\Lambda$ be defined as in \eqref{e:def:escoupling}. Then:
\begin{itemize}
\item The marginal of $\mu^\xi_\Lambda$ on $\Sigma_\Lambda^\xi$ is $\pi^{\xi}_{\Lambda}$.
\item The marginal of $\mu^\xi_\Lambda$ on $\Omega$ is $\phi^\xi_{\Lambda,p,2}$.
\end{itemize}
\end{lem}

In order to characterize the regularity of the Ising magnetization field $\Phi_a$ on an unbounded domain $U \subseteq \R^2$, in the next sections we will use the well-known FK-Ising coupling for infinite volume measures.
\begin{thm}[{\cite[Theorem~4.91]{grim}}] \label{t:escoupling:infinite}
Let $p \in [0,1]$, $q = 2$, $e^{-\beta} = (1-p)$.
\begin{itemize}
\item Let $\omega$ be sampled from $\Omega = \{ 0,1 \}^{\E^2}$ with law $\phi^1_{p,q}$. Conditional on $\omega$, each vertex is assigned a random spin $\sigma_x \in \{ -1,+1 \}$ such that:
\begin{enumerate}
\item $\sigma_x =1 $ if  $x \leftrightarrow \infty$
\item $\sigma_x$ takes values in $\{ -1, 1 \}$ with probability $\frac{1}{2}$ if $x \nleftrightarrow \infty$
\item $\sigma_x = \sigma_y $ if $x \leftrightarrow y$
\item spins in different open clusters are independent.
\end{enumerate}
Then the configuration $\sigma = \{ \sigma_x \}_{x \in \Z^d}$ is distributed according to the weak limit $\pi^1 $ of Ising measures with $+$ boundary condition.
\item Let $\sigma $ be sampled from $\Sigma = \{ -1, +1 \}^{\Z^d}$ with the Ising limit law $\pi^1$.  Conditional on $\sigma$, each edge is assigned a random state $\omega_e \in \{ 0,1 \}$ such that:
\begin{enumerate}
\item the states of different edges are independent
\item $\omega_e = 0$ if $\sigma_x \neq \sigma_y$
\item if $\sigma_x = \sigma_y$, then $\omega_e = 1$ with probability $p$ and $0$ otherwise.
\end{enumerate}
Then the edge configuration $\omega = \{ \omega_e \}_{e \in \E^2}$ has law $\phi^1_{p,q}$.
\end{itemize}
\end{thm}
A similar argument is valid for $\phi^0_{p,q}$ and the infinite-volume Ising measure $\pi^0$, with the difference that no fixed value is assigned to $\sigma_x$ in the case $x \leftrightarrow \infty$.

%

\subsection{Tightness of the Ising magnetization field} \label{s:mag}
We now consider the planar Ising magnetization field at critical temperature $\beta_c$, on an open set $U \subset \R^2$ (possibly unbounded or equal to $\R^2$). Call $U_a = U \cap a \Z^2$ for $a > 0, a \in \R$. As in \cite{cgn1} we define an approximation of the Ising magnetization field at scale $a > 0$ as
\begin{equation} \label{e:def:isingmag}
\Phi_a := a^{-\frac 1 8} \sum_{y \in U_a} \sigma_y  \, \mathbbm{1}_{S_a(y)} ,
\end{equation}
where $S_a(y)$ is the (open) square centered at $y$ of side-length $a$, and $\sigma_y$ is the Ising spin at $y$. 

We investigate this quantity at critical temperature, with either $+$ or free boundary condition on $U_a$. Our aim is to establish its tightness in $\B^{\alpha,\mathrm{loc}}_{p,q}(U)$. In order to do that, we will choose a spanning sequence $\msc K$ of $U$ and bound \eqref{e:tight1}, \eqref{e:tight2} for $\Phi_a$, which if $p$ is even become
\begin{equation} \label{e:phitight1p}
a^{-\frac{1}{8}} \sup_{x \in \Lambda_k \cap K}  \left[ \sum_{y_1\ldots y_p \in U_a}  \E^\xi_{U_a}(\sigma_{y_1}\cdots \sigma_{y_p})  \prod_{j=1}^{p}  \int_{S_a(y_j)}   \varphi (2^k (z - x)) \d z   \right]^{\frac{1}{p}},
\end{equation}
\begin{equation}\label{e:phitight2p}
a^{-\frac{1}{8}}  2^{2 n} \sup_{x \in \Lambda_n \cap K}  \left[ \sum_{y_1\ldots y_p \in U_a}  \E^\xi_{U_a}(\sigma_{y_1} \cdots \sigma_{y_p})  \prod_{j=1}^{p} \int_{S_a(y_j)} \psii (2^n (z - x)) \d z  \right]^{\frac{1}{p}} ,
\end{equation}
with $(K,k) \in \msc K$.
Here $\E^\xi_{U_a}(\sigma_{y_1} \cdots \sigma_{y_p})$ is the expectation with respect to the Ising-Potts measure $\pi^\xi_{U_a}$ at critical temperature with either free or $+$ boundary condition (see~\eqref{e:def:isingfree} and \eqref{e:def:isingplus}). 

In the following discussion we will exploit the Ising-FK relation discussed in Subsection~\ref{s:escoupling} and introduce some lemmas which are useful to prove Theorem~\ref{t.Ising}. 

Let $\Lambda \subset \Z^2$ be a finite set. 
Define $A^1_{y_1 \ldots y_n} \subset \{ 0,1 \}^{E_\Lambda} $ the event that each open cluster of the FK model on $\Lambda$ contains an even number of the points $ y_1 , \ldots , y_n$, or is connected to the boundary $\partial \Lambda$. Define also $A^{1,\infty}_{y_1 \ldots y_n} \subset \{ 0,1 \}^{\E^2} $ the event that each  open cluster of the FK model on $\Z^2$ contains an even number of the points $ y_1 , \ldots , y_n$,  or is infinite.
Finally, let $A^0_{y_1 \ldots y_n}$ be the event that each open cluster contains an even number of the points $ y_1 , \ldots , y_n$. It is easy to notice that all these events are increasing, i.e. they are preserved when switching any $\omega_e$ from $0$ to $1$.


\begin{lem} \label{l:fk-ising}
Let $\phi$ be the FK probability measure with $p \in [0,1]$ and $	q=2$, and take~$e^{-\beta} = (1-p)$. Then for any $n \geqslant 1$:
\begin{enumerate}

\item
$
\E^+_{\Lambda}(\sigma_{y_1} \cdots \sigma_{y_n}) = \phi_{\Lambda}^1 (A^1_{y_1 \ldots y_n}).
$

\item
$
\E^+_{\Z^2}(\sigma_{y_1} \cdots \sigma_{y_n}) = \phi_{\Z^2}^1 (A^{1,\infty}_{y_1 \ldots y_n}).
$

\item
$
\E^{\mathrm{free}}_{\Lambda}(\sigma_{y_1} \cdots \sigma_{y_n}) = \phi_{\Lambda} (A^0_{y_1 \ldots y_n}).
$

\item
$
\E^{\mathrm{free}}_{\Z^2}(\sigma_{y_1} \cdots \sigma_{y_n}) = \phi^0_{\Z^2} (A^0_{y_1 \ldots y_n})
$

\end{enumerate}
\end{lem}
\begin{proof}
We only prove the first point in this lemma, as the other equalities can be obtained with the same arguments, using Theorem~\ref{t:escoupling:infinite}.
Let $f(\sigma)=\sigma_{y_1} \cdots \sigma_{y_n}$, from Lemma~\ref{l:escoupling} and \eqref{e:def:escoupling} we can write
\begin{align*}
\E^+_{\Lambda} [f (\sigma) ] & = \sum_{\sigma \in \Sigma^1_\Lambda} f(\sigma) \sum_{\omega \in \Omega} \mu^1_\Lambda (\sigma, \omega) \\
& = \frac{1}{Z^1_{ES}} \sum_{\sigma \in \Sigma^1_\Lambda} f(\sigma) \sum_{\omega \in \Omega^1_\Lambda} \prod_{e \in E_\Lambda} \left[ (1-p) \mathbbm{1}_{\omega_e = 0} + p \mathbbm{1}_{\omega_e = 1} \mathbbm{1}_{\sigma_e = 1}   \right]  \\
& =  \frac{1}{Z^1_{ES}} \sum_{\omega \in \Omega^1_\Lambda}   (1-p)^{| E_\Lambda \setminus \eta_\Lambda (\omega)|} p^{|\eta_\Lambda (\omega)|}  \sum_{\sigma \in \Sigma^1_\Lambda} f(\sigma) \prod_{e \in \eta_\Lambda (\omega)} \mathbbm{1}_{\sigma_e = 1} \\
\end{align*}
with $\eta_\Lambda (\omega) = \{ e \in E_\lambda \mid \omega_e =1 \}$.

Now take $\omega \in \Omega^1_\Lambda$ such that one or more of its clusters contain an odd number of points in $y_1 \ldots y_n$. The sum $\sum_{\sigma \in \Sigma^1_\Lambda} f(\sigma) \prod_{e \in \eta_\Lambda (\omega)} \mathbbm{1}_{\sigma_e = 1}$ is zero (indeed, each odd cluster takes the values $+1$ and $-1$ and all terms cancel out). Conversely, if $\omega \in A^1_{y_1 \ldots y_n}$, the product $\sigma_{y_1} \cdots \sigma_{y_{2k}}$ in the same cluster is equal to $1$.  We can write then:

\begin{align*}
\E^+_{\Lambda} [f (\sigma) ] & =  \frac{1}{Z^1_{ES}} \sum_{\omega \in \Omega^1_\Lambda} \mathbbm{1}_{A^1_{y_1 \ldots y_n}}(\omega)   (1-p)^{| E_\Lambda \setminus \eta_\Lambda (\omega)|} p^{|\eta_\Lambda (\omega)|}  \sum_{\sigma \in \Sigma^1_\Lambda}  \prod_{e \in \eta_\Lambda (\omega)} \mathbbm{1}_{\sigma_e = 1}  \\
& = \frac{1}{Z^+_{ES}} \sum_{\omega \in \Omega^1_\Lambda} \mathbbm{1}_{A^1_{y_1 \ldots y_n}}(\omega)   (1-p)^{| E_\Lambda \setminus \eta_\Lambda (\omega)|} p^{|\eta_\Lambda (\omega)|} 2^{\overline{k(\omega,\E^2 \setminus E_\Lambda)}}
\end{align*}
Here $\overline{k(\omega,\E^2 \setminus E_\Lambda)}$ is the number of connected clusters of $\omega$ that do not intersect $\E^2 \setminus E_\Lambda$. 

The following equivalence between partition functions yields the result:
\begin{align*}
Z^1_{ES} & = \sum_{\omega \in \Omega^1_\Lambda}   (1-p)^{| E_\Lambda \setminus \eta_\Lambda (\omega)|} p^{|\eta (\omega)|}  \sum_{\sigma \in \Sigma^1_\Lambda} \prod_{e \in \eta_\Lambda (\omega)} \mathbbm{1}_{\sigma_e = 1} \\
& = \frac{1}{2} \sum_{\omega \in \Omega^1_\Lambda}   (1-p)^{| E_\Lambda \setminus \eta_\Lambda (\omega)|} p^{|\eta_\Lambda  (\omega)|} 2^{k(\omega, E_\Lambda)} = \frac{1}{2}  Z^{1,\Lambda}_{FK}(p,2). \qedhere
\end{align*}  
\end{proof}

We are going to need a well-known inequality for the 2-d Ising model of Onsager, formulated using connection probabilities for the FK model. See also \cite[Lemma~5.4]{d-chn}.
\begin{lem} \label{l:onsager}
Let $m \in \N$ and $B_{m} = [-m,m]^2 \cap \Z^2$ . At critical temperature $p_c = 1- e^{-\beta_c}$, there exists $C >0 $ such that:
$$
\phi^1_{B_m,p_c,q=2} (0 \leftrightarrow \partial B_m) \leqslant C m^{-\frac{1}{8}}.
$$
\end{lem}

The following proposition is known (see \cite[Proposition~3.9]{cgn1} for a sketch of the proof), but we give here a different (and complete) proof which employs \emph{the pin and sum argument with hairy cycles} of \cite{Abdesselam}. 

\begin{prop} \label{p:garbound}
Let $p \in \N$. There exists $C > 0$ such that, for every $N \in \N$:
\begin{equation}
\label{e:garbound}
\sum_{y_1, \ldots , y_p \in U_N} \E^\xi_{U_N \left( \Z^2  \right)} (\sigma_{y_1} \cdots \sigma_{y_p} ) \leqslant C (N+1)^{\frac{15}{8}p}
\end{equation}
with $U_N = [0,N]^2 \cap \Z^2 $ and $\E^\xi_{U_N \left( \Z^2  \right)}$ being the expectation on either $U_N$ or $\Z^2$ at critical temperature $\beta_c$.
\end{prop}
\begin{proof}
The events $A_{y_1 \ldots y_p}$ are increasing, and we have $A^0_{y_1 \ldots y_p} \subset A^1_{y_1 \ldots y_p}$ when the events are on the same domain (finite or infinite). From the coupling of Lemma~\ref{l:fk-ising}, and using the monotonicity properties of Lemma~\ref{l:fkg} it is easy to obtain $\E^\xi_{U_N \left( \Z^2  \right)} \leqslant  \E^+_{U_N}$. We are then left to show the inequality for this term.

We start by showing that 
\begin{equation} \label{e:corr-alldiff}
\sum_{\substack {y_1, \ldots , y_p \in U_N \\ y_i \neq y_j \forall i \neq j }} \E^+_{U_N} (\sigma_{y_1} \cdots \sigma_{y_p} ) \leqslant C N^{\frac{15}{8}p}.
\end{equation}
The event $A^1_{y_1 \ldots y_p}$ of Lemma~\ref{l:fk-ising} implies that every point in $\{ y_1 , \ldots , y_p \}$ is connected by an open path to another point in $\{ y_1 , \ldots , y_p \}$ or to the boundary $\partial U_N$, which we call $y_0$. For every $1 \leqslant i \leqslant p$, call $\ell_i = \min_{j \geqslant 0, j \neq i} \mathrm{d}(y_i,y_j)$ where $\mathrm{d}(y_i,y_j)$ is the $\Z^2$ distance between $y_i$ and $y_j$, and define $B_i = y_i + \llbracket -\ell_i /4 , \ell_i /4 \rrbracket^2 $, $F = \bigcup_{i = 1}^p  B_i$. Notice that the graph $F \subset \Z^2$ has $p$ disjoint components. 

From Lemma~\ref{l:domainmarkov}, Remark~\ref{r:indep} and since $\phi^+_{U_N} (A) = \sum_{\omega} \phi^+_{U_N} (A \mid \mathcal{T}_F) (\omega ) \phi^+_{U_N} (\omega)$, we obtain
$$
\E^+_{U_N} (\sigma_{y_1} \cdots \sigma_{y_p} ) \leqslant  \phi^+_{U_N} \left( \bigcap_{i = 1}^p \{ y_i \leftrightarrow \partial B_i  \} \right) \leqslant \prod_{i = 1}^p \phi^+_{B_i} (y_i
   \leftrightarrow \partial B_i) ,
$$
where we used the monotonicity property of Lemma~\ref{l:fkg} in the second inequality. Lemma~\ref{l:onsager} yields:
\begin{align*}
\sum_{\substack {y_1, \ldots , y_p \in U_N \\ y_i \neq y_j \forall i \neq j }} \E^+_{U_N} (\sigma_{y_1} \cdots \sigma_{y_p} ) & \lesssim  \sum_{\substack {y_1, \ldots , y_p \in U_N \\ y_i \neq y_j \forall i \neq j }} \prod_{i = 1}^p \left[ \min_{j \geqslant 0, j \neq i} \mathrm{d}(y_i,y_j)  \right]^{- \frac{1}{8}} \\
& \lesssim \sum_{\substack {y_1, \ldots , y_p \in U_N \\ y_i \neq y_j \forall i \neq j }} \prod_{i = 1}^p \sum_{\substack{j = 0 \\ j \neq i}}^p \mathrm{d}(y_i,y_j)^{- \frac{1}{8}} \\
& \lesssim \sum_{\substack {j_1 \ldots j_p = 0 \\  j_i \neq i}}^p \sum_{\substack {y_1, \ldots , y_p \in U_N \\ y_i \neq y_j \forall i \neq j }} \mathrm{d}(y_1, y_{j_1})^{- \frac{1}{8}} \cdot \ldots \cdot \mathrm{d}(y_p, y_{j_p})^{- \frac{1}{8}}  
\end{align*}

It is easy to see that for $i \in \{ 1 \ldots p \}$, $j \in \{ 0 \ldots p \}$
\begin{equation}
\sum_{\substack{y_i \in U_N \\ y_i \neq y_j }} \mathrm{d}(y_i,y_j)^{-\frac{1}{8}} \lesssim N^{\frac{15}{8}} , \label{e:onedist}
\end{equation} 
there are indeed $\sim k$ points at distance $k$ from $y_{j}$. 

To estimate the term
\begin{equation}
\sum_{\substack {y_1, \ldots , y_p \in U_N \\ y_i \neq y_j \forall i \neq j }} \mathrm{d}(y_1, y_{j_1})^{- \frac{1}{8}} \cdots \mathrm{d}(y_p, y_{j_p})^{- \frac{1}{8}} 
 \qquad (0 \leqslant j_i \leqslant p, j_i \neq i)  \label{e:manydist}
\end{equation}  
we need to find the right order in which to compute the sums $\sum_{y_i}$. 
We associate then \eqref{e:manydist} to a graph with $p+1$ vertices $\{0, 1, \ldots , p \}$ and $p$ directed edges, 
 such that to $\mathrm{d}(y_i, y_{j_i})$ corresponds an edge going from $i$ to $j_i$. 
 
Notice that every vertex in $\{ 1, \ldots, p \}$
has exactly one edge going to a vertex in $\{ 0,1, \ldots , p \}$ and the vertex $0$ has no outgoing edges. Therefore, following the directed edges starting from any vertex in $\{ 1 ,\ldots , p\}$ one either ends up at the vertex $0$, or enters a cycle (because every vertex except $0$ has an outgoing edge). This cycle cannot be escaped, again because vertices in $\{ 1 ,\ldots , p\}$  have only one outgoing edge (indeed, to every $y_i$ there is only one $y_{j_i}$ associated to it). 

This said, we can conclude that our graph has one or more connected components, each of which can be of two distinct types:
\begin{itemize}
\item a tree with root in the vertex $0$
\item a cycle, possibly with branches attached to it (i.e. each point of the cycle can be the root of a tree).
 \end{itemize}

We can then proceed to estimate every sum in~\eqref{e:manydist} in the order given by the oriented
graph, starting from the leaves. This is just a repeated application of~\eqref{e:onedist}, until we reach the root ($0$) or a circle. Hence every connected component with root in $0$ and $k$ edges gives a term of order $N^{\frac{15}{8} k}$.
For example we can estimate the following term as follows (starting from the leaves $y_1$ and $y_3$):
\begin{multline*} 
\sum_{\substack{y_1, y_2, y_3, \in U_N \\ y_i \neq y_j \forall i \neq j }} \mathrm{d}(y_1,y_2)^{- \frac{1}{8}} \mathrm{d}(y_2,
   y_0)^{- \frac{1}{8}} \mathrm{d}(y_3,y_2)^{- \frac{1}{8}} \\
   \leqslant \sum_{y_2 \in U_N} \mathrm{d}(y_2,y_0)^{- \frac{1}{8}} \sum_{\substack{y_1 \in U_N \\ y_1 \neq y_2 }} \mathrm{d}(y_1,y_2)^{- \frac{1}{8}}
    \sum_{\substack{y_3 \in U_N \\ y_3 \neq y_2}} \mathrm{d}(y_2,y_3)^{- \frac{1}{8}} \lesssim N^{\frac{45}{8}}   
\end{multline*}

Summing on circles does not pose any additional problem: indeed one can just
choose a point within the circle (call it $\hat{y}_2$) and sum keeping fixed
both the ``inbound'' point $\hat{y}_1$ and the ``outbound'' point $\hat{y}_3$:
\begin{align*}
\sum_{\substack{\hat{y}_2 \in U_N \\ \hat{y}_2 \neq \hat{y}_1, \hat{y}_2 \neq \hat{y}_3}} \mathrm{d}(\hat{y}_1, \hat{y}_2)^{- \frac{1}{8}} \mathrm{d}
   (\hat{y}_2, \hat{y}_3)^{- \frac{1}{8}} 
&  \leqslant \sum_{\substack{\hat{y}_2 \in
   U_N \\ \hat{y}_2 \neq \hat{y}_1  }} \frac{ \mathrm{d}(\hat{x}_1, \hat{x}_2)^{- \frac{1}{4}}}{2} +
   \sum_{\substack{\hat{y}_2 \in
   U_N \\ \hat{y}_2 \neq \hat{y}_3  }} \frac{ \mathrm{d}(\hat{y}_2, \hat{y}_3)^{-
   \frac{1}{4}}}{2} \\
& \lesssim N^{2 - \frac{1}{4}} 
\end{align*}
where we used Young inequality. Then (for a circle with $k$ edges) the sum
over the remaining vertices $\hat{y}_3 \ldots \hat{y}_k$ gives an estimation of order
$N^{\frac{15}{8} (k - 2)}$. This proves~\eqref{e:corr-alldiff}.

Now consider the general case in which two or more points concide. At the price of a factor $p!$ we can reorder the points, and take the last $p - k$ points to be all different from each other (with $2 \leqslant k \leqslant p$). Conversely, $\{ y_1, \ldots , y_k \}$ can be partitioned in $m$ subsets such that all the points in the same subset are equal: we call $k_i$ the number of points in the $i$-th subset with $k = k_1 + \ldots +k_m$, and therefore $m \leqslant k/2$. We want to show that:
$$
\sum_{\substack{\bar{y}_1, \ldots , \bar{y}_m \in U_N,   \\ \bar{y}_i \neq y_j, i \leqslant m , j \in [k+1, p]}} \sum_{\substack {y_{k+1}, \ldots , y_p \in U_N \\ y_i \neq y_j  }} \E^+_{U_N} (\sigma_{\bar{y}_1}^{k_1} \cdots \sigma_{\bar{y}_m}^{k_m} \sigma_{y_{k+1}} \cdots \sigma_{y_p} ) \leqslant C N^{\frac{15}{8}p}.
$$ 
As before we define $\ell_i = \min_{j \geqslant 0, j \neq i} \mathrm{d}(y_i,y_j)$  for every $k+1 \leqslant i \leqslant p$ and $B_i = y_i + \llbracket -\ell_i /4 , \ell_i /4 \rrbracket^2 $. Notice that the event $A^1_{y_1 \ldots y_p}$ implies that every $y_i$ with $i \geqslant k+1$ is connected by an open path to the boundary of $B_i$. Then using the results already obtained:
\begin{multline*}
\sum_{\substack{\bar{y}_1, \ldots , \bar{y}_m \in U_N  \\\bar{y}_i \neq y_j , i \leqslant m , k+1 \leqslant j \leqslant p}} \sum_{\substack {y_{k+1}, \ldots , y_p \in U_N \\ y_i \neq y_j  }} \E^+_{U_N} (\sigma_{\bar{y}_1}^{k_1} \cdots \sigma_{\bar{y}_m}^{k_m} \sigma_{y_{k+1}} \cdots \sigma_{y_p} ) \\
\lesssim N^{2 m} \phi^+_{U_N} \left( \bigcap_{i = k+1}^p \{ y_i \leftrightarrow \partial B_i  \} \right) 
 \lesssim N^{2 m} N^{\frac{15}{8}(p-k)} \leqslant N^{\frac{15}{8}p}.
\end{multline*}
\end{proof}

We are now ready to prove Theorem~\ref{t.Ising}.

\begin{proof}[Proof of Theorem~\ref{t.Ising}]
By Theorem~\ref{t:tight}, the result is proved as soon as we can bound \eqref{e:phitight1p} and \eqref{e:phitight2p} for any even $p \geqslant 2$. If the domain $U$ is bounded, we choose $\msc K = (K_n,n)_{n \in \N}$ as its spanning sequence, with:
\begin{equation}\label{e:newspan} 
K_n = \{ x \in \R^2  \mid \mathrm{dist}(x,U^c) \geq (2+\delta) R 2^{-n}  \}
\end{equation}
for $\delta >0$ and $R$ such that \eqref{e:def:R} holds. If $U$ is unbounded, it suffices to take $\hat{K}_n = K_n \cap \bar{B}(0,n)$: in both cases we have a valid spanning sequence according to Definition~\ref{def:span}.

We first consider \eqref{e:phitight2p}.
From the support properties of  $\psii (2^n (\cdot- x))$ \eqref{e:def:R} we can restrict the sum over $y_j$ to the set 
$$\Omega_{n,x} = \{ y \in  U_a \mid \mathrm{d}(y,x) < 2^{-n}R + a/\sqrt{2} \}.$$
Now we bound \eqref{e:phitight2p} separately for small and large values of $n$.

If~$2^n \geqslant R  a^{-1}$ we have
\begin{multline*}
\sum_{y_1\ldots y_p \in U_a}  \E^\xi_{U_a}(\sigma_{y_1} \cdots \sigma_{y_p})  \prod_{j=1}^{p} \int_{S_a(y_j)} \psii (2^n (z - x)) \d z  \\
\leqslant \sum_{y_1\ldots y_p \in \Omega_{n,x}} \prod_{j=1}^{p} \int_{S_a(y_j)} \left| \psii (2^n (z - x)) \right| \d z   
\leqslant  \sum_{y_1\ldots y_p \in \Omega_{n,x}}2^{-2 p n}  \lesssim 2^{-2 p n}.
\end{multline*}
This gives the estimation
$$
a^{-\frac{1}{8}}  2^{2 n} \sup_{x \in \Lambda_n \cap K}  \left[ \sum_{y_1\ldots y_p \in U_a}  \E^\xi_{U_a}(\sigma_{y_1} \cdots \sigma_{y_p})  \prod_{j=1}^{p} \int_{S_a(y_j)} \psii (2^n (z - x)) \d z  \right]^{\frac{1}{p}} \lesssim 2^{\frac 1 8 n}.
$$

Conversely, if~$2^n < R  a^{-1}$ we first notice that 
$$
\Omega_{n,x} \subset  \tilde{U}_{a,x} = [x- 2R 2^{-n}, x+ 2R 2^{-n}]^2 \cap a \Z^2
$$
and then using Lemma~\ref{l:fkg}:
\begin{multline*}
 \sum_{y_1\ldots y_p \in \tilde{U}_{a,x}} \E^{\xi}_{U_a}(\sigma_{y_1} \cdots \sigma_{y_p})
  \lesssim \sum_{y_1\ldots y_p \in \tilde{U}_{a,x}} \E^+_{U_a}(\sigma_{y_1} \cdots \sigma_{y_p}) \\
  \lesssim  \sum_{y_1\ldots y_p \in \tilde{U}_{a,x}} \E^+_{\tilde{U}_{a,x}}(\sigma_{y_1} \cdots \sigma_{y_p}) 
  \lesssim \sum_{y_1\ldots y_p \in \llbracket -N , N \rrbracket ^2} \E^+_{\llbracket -N , N \rrbracket ^2}(\sigma_{y_1} \cdots \sigma_{y_p})
\end{multline*}
with $N = \lfloor  \frac{2R 2^{-n}}{a} \rfloor$. By Proposition~\ref{p:garbound}, we finally obtain
$$
\sum_{y_1\ldots y_p \in \tilde{U}_{a,x}} \E^{\xi}_{U_a}(\sigma_{y_1} \cdots \sigma_{y_p}) \lesssim a^{-\frac{15}{8}p} 2^{-\frac{15}{8}pn},
$$
uniformly over $x$. As a result, \eqref{e:phitight2p} can be bound from above by $C 2^{\frac{1}{8}n}$ for some~$C>0$.
Using the same techniques it is easy to obtain a bound for \eqref{e:phitight1p}:
$$
a^{-\frac{1}{8}} \sup_{x \in \Lambda_k \cap K}  \left[ \sum_{y_1\ldots y_p \in U_a}  \E^\xi_{U_a}(\sigma_{y_1}\cdots \sigma_{y_p})  \prod_{j=1}^{p}  \int_{S_a(y_j)}   \varphi (2^k (z - x)) \d z   \right]^{\frac{1}{p}} \lesssim 1 .
$$

Therefore, by the tightness criterion of Theorem~\ref{t:tight} we have shown that $\Phi_a$ is tight in $\B^{-\frac{1}{8} - \varepsilon, \mathrm{loc}}_{p,q}(U)$ for $p \geqslant 2$ and even. The embedding described in Remark~\ref{r:easy.embed} yields the result for all $p \in [1,\infty]$.
\end{proof}

\subsection{Absence of tightness in higher-order spaces} \label{s:nomag}

In this subsection, we prove Theorem~\ref{t.Ising.converse}.
The proof is based on the following lemma, which is a consequence of the RSW-type bounds for the FK model obtained in \cite{d-chn}.
\begin{lem}[{\cite[Proposition~27]{d-chn}}] \label{l:corr-lowbound}
There exists $c>0$ such that for any $y_1,y_2 \in \Z^2$ with $\mathrm{d}(y_1,y_2) > 0$:
$$
\E^\xi_{\Z^2} (\sigma_{y_1} \sigma_{y_2} ) \geqslant c\,  \mathrm{d}(y_1,y_2)^{-\frac{1}{4}}
$$
for any boundary condition $\xi$.
\end{lem}

In order to show the absence of tightness we only need the following partial converse to Proposition~\ref{p:garbound} for two-points correlations.
\begin{lem} \label{l:garbound-converse}
There exists $c > 0$ such that, for every $N \in \N$:
\begin{equation*}
\sum_{y_1, y_2 \in U_N} \E^\xi_{\Z^2} (\sigma_{y_1} \sigma_{y_2} ) \geqslant c (N+1)^{\frac{15}{4}}
\end{equation*}
with $U_N = [0,N]^2 \cap \Z^2 $ and $\E^\xi_{ \Z^2 }$ being the expectation on $\Z^2$ with arbitrary boundary conditions.
\end{lem}
\begin{proof}
The result is immediate since there are $(N+1)^4$ terms in the sum, each being larger than $c (N+1)^{-\frac{1}{4}}$ for some fixed constant $c>0$.
\end{proof}

We now present an equivalent norm $\mathscr{E}^{\alpha}_p$ for Besov spaces, which reduces to Definition~\ref{def.Besov} in the case $p =\infty$.
\begin{defi}[{\cite[Definition~2.5]{HaLa15}}] 
\label{def.hala}
Let $f \in C^\infty_c$. For every $\alpha < 0$ and $p \in [1,\infty]$ we introduce the norm
$$
\| f \|_{\mathscr{E}^{\alpha}_p} :=
\sup_{\lambda \in (0,1]} \lambda^{-\alpha} \Ll\|  \sup_{\eta \in \mathscr{B}_{r_0}}  \left| \langle f , \eta_{\lambda,x} \rangle
 \right| \Rr\|_{L^p (\d x)} 
$$ 
with $\eta_{\lambda,x} := \lambda^{-d} \eta \Ll( \lambda^{-1} (\cdot - x) \Rr)$ and $\mathscr{B}_{r_0}$ as in Definition~\ref{def.Besov}.
\end{defi}
The following is a straightforward generalization of Proposition~\ref{p.equiv}.
\begin{lem}[{\cite[Proposition~2.6]{HaLa15}}] \label{p.equiv-general}
Let $\alpha < 0$. There exist $C_1,C_2 \in (0,\infty)$ such that for every $f \in C^\infty_c$, we have
\begin{equation}
C_1 \|f\|_{\mathscr{E}^{\alpha}_p} \le \|f\|_{\B^{\al}_{p,\infty}} \le C_2 \|f\|_{\mathscr{E}^{\alpha}_p}.
\end{equation}
\end{lem}
The advantage of Definition~\ref{def.hala} is that it allows us to easily obtain lower bounds on the Besov norm of some distribution by testing against a non-negative function. We can now proceed to prove Theorem~\ref{t.Ising.converse}.
\begin{proof}[Proof of Theorem~\ref{t.Ising.converse}]
We decompose the proof into three steps.

\smallskip

\emph{Step 1.} In this first step, we recall that for a non-negative random variable $X$, we have
\begin{equation} \label{e:inverse-markov}
\P\Ll[X > \frac{\E[X]}{2}\Rr] \ge \frac {(\E[X])^2}{4 \E[X^2]}.
\end{equation} 
Indeed, this follows from 
\begin{equation*}  
\E[X]  = \E[X \1_{X \le \E[X]/2}] + \E[X \1_{X > \E[X]/2}] \le \frac {\E[X]} 2  + \E\Ll[X^2\Rr]^\frac 1 2 \P\Ll[X > \frac{\E[X]}{2}\Rr]^{\frac 1 2},
\end{equation*}
by the Cauchy-Schwarz inequality.




\smallskip

\emph{Step 2.}
Let $\eta$ be a smooth non-negative function supported on the ball $B(0,1)$ and such that $\eta \equiv 1$ on $B(0,1/2)$. We set $\eta_{\lambda, x} := \lambda^{-2} \eta (\lambda^{-1} (\cdot -x))$ and
\begin{equation*}  
X_{a,\lambda} := \int_{B(0,1)} \Ll| \langle \Phi_a, \eta_{\lambda,x} \rangle \Rr| \, \d x.
\end{equation*}
In this step, we show that there exists a constant $c > 0$ such that for every $a < \lambda \in (0,1]$,
\begin{equation}  
\label{e.main.step2}
\P\Ll[X_{a,\lambda} \ge c \lambda^{-\frac 1 8} \Rr] \ge c.
\end{equation}
As in the proof of Theorem~\ref{t.Ising}, we can use Proposition~\ref{p:garbound} to show that there exists a constant $C < \infty$ such that for every $p \in \{2,4\}$, $a < \lambda \in (0,1]$ and $x \in \R^2$,
\begin{equation} 
\label{e.hello.again}
\E \Ll[ \Ll(\langle \Phi_a, \eta_{\lambda,x}\rangle\Rr)^p \Rr]  \le C \, \lambda^{-\frac p 8}.
\end{equation}
By a similar reasoning, we obtain from Lemma~\ref{l:garbound-converse} that there exists a constant $c > 0$ such that for every $a < \lambda \in (0,1]$  and $x \in \R^2$,
\begin{equation}  
\label{e.lowerguy}
\E \Ll[ \Ll(\langle \Phi_a, \eta_{\lambda,x}\rangle\Rr)^2 \Rr]  \ge  c \, \lambda^{-\frac 1 4}.
\end{equation}
Combining \eqref{e:inverse-markov}, \eqref{e.hello.again} with $p = 4$ and \eqref{e.lowerguy}, we deduce that for every $a < \lambda \in (0,1]$  and $x \in \R^2$,
\begin{equation*}  
\P \Ll[ \Ll|\langle \Phi_a, \eta_{\lambda,x}\rangle\Rr|\ge \frac{\sqrt{c}}{2} \lambda^{-\frac 1 8} \Rr] \ge \frac c C \;.
\end{equation*}
In particular, after reducing the constant $c > 0$ as necessary, we obtain that for every $a < \lambda \in (0,1]$ and $x \in \R^2$,
\begin{equation}
\label{e.estim.step2}
\E \Ll[ \Ll| \langle \Phi_a, \eta_{\lambda,x} \rangle \Rr|  \Rr] \ge c \lambda^{-\frac 1 8},
\end{equation}
and thus that
\begin{equation*}  
\E[X_{a,\lambda}] \ge c \lambda^{-\frac 1 8}.
\end{equation*}
Using \eqref{e.hello.again} with $p = 2$ and Jensen's inequality, we also have, for every $a < \lambda \in (0,1]$,
\begin{equation*}  
\E[X_{a,\lambda}^2] \le C \lambda^{-\frac 1 4}.
\end{equation*}
We therefore obtain \eqref{e.main.step2} by another application of \eqref{e:inverse-markov}.

\smallskip

\emph{Step 3.}
Let $\al > -\frac 1 8$, and let $\ov \Phi$ be a possible limit point of the family $(\Phi_a)_{a \in (0,1]}$. Passing to the limit along a subsequence in \eqref{e.main.step2}, we get that for every $\lambda \in (0,1]$,
\begin{equation}  
\label{e.estim.step3}
\P \Ll[ \int_{B(0,1)} \Ll|\langle \ov \Phi , \eta_{\lambda,x}\rangle\Rr| \, \d x  \ge c \lambda^{-\frac 1 8} \Rr] \ge c  .
\end{equation}
By Remark~\ref{r:easy.embed}, in order to prove Theorem~\ref{t.Ising.converse}, it suffices to show that $\ov \Phi \notin \Bl{\al}_{1,\infty}(\R^2)$ with positive probability. Let $\chi$ be a non-negative smooth function of compact support such that $\chi \equiv 1$ on $B(0,2)$. By Lemma~\ref{p.equiv-general}, there exists a constant $c' > 0$ such that for every $\lambda \in (0,1]$,
\begin{align*}
\|\chi \ov \Phi\|_{\B^\al_{1,\infty}} & \ge c' \lambda^{-\alpha} \int_{\R^2} \left| \langle \chi \bar{\Phi} , \eta_{\lambda,x} \rangle \right| \d x \ge  c' \lambda^{-\alpha} \int_{B(0,1)} \left| \langle \bar{\Phi} , \eta_{\lambda,x} \rangle \right| \d x.
\end{align*}
Combining this with \eqref{e.estim.step3} yields
\begin{equation*}  
\P \Ll[ \|\chi \ov \Phi\|_{\B^\al_{1,\infty}} \ge c c' \lambda^{-\al -\frac 1 8} \Rr] \ge c .
\end{equation*}
Since $\al > -\frac 1 8$, letting $\lambda$ tend to $0$ gives
\begin{equation*}  
\P \Ll[ \|\chi \ov \Phi\|_{\B^\al_{1,\infty}} = +\infty \Rr] \ge c  > 0,
\end{equation*}
which completes the proof.
\end{proof}

\appendix
\section{} \label{a:notselfcont}

In Section~\ref{s:tightcrit} we left behind some details for the sake of self-containedness: in particular the proof of Proposition~\ref{p:mult-2} with Besov spaces of the type $\Bbl$ for any $p,q \geqslant 1$. In order to show that this statement is true in the general case (and not only for $\B^{\alpha,\mathrm{loc}}_{\infty,\infty}$) we need some results about the product of elements of Besov spaces. We obtain these by relating the Besov spaces as defined in this paper with those in \cite{bcd}, defined via the Littlewood-Paley decomposition.

\begin{thm}[{\cite[Proposition~2.9.4]{mey}}] \label{a:t:multimeyer}
Let $\alpha > 0$, $p,q \in [1,\infty]$ and $f \in L^p(\R^d)$. The following two properties are equivalent.
\begin{enumerate}
\item Let $r > \alpha$ be an integer and $\phi,  (V_n)_{n \in \Z}$ be a $r$-regular multiresolution analysis as of Definition~\ref{def:multi}. Then the sequence $2^{n \alpha} \| \W_n f \|_{L^p}$ belongs to $\ell^q (\N )$ and $ \V_0 f $ belongs to $L^p(\R^d)$.
\item There exists a sequence of positive numbers $\varepsilon_n \in \ell^q (\N )$ and a sequence of functions $f_0 , g_0, g_1, \ldots \in L^p (\R^d )$ such that $f = f_0 + \sum_{n \geqslant 0} g_n$,  $\| g_n \|_{L^p} \leqslant \varepsilon_n 2^{- n \alpha}$ for $n \geqslant 0$ and $\| \partial^{k} g_n \|_{L^p} \leqslant \varepsilon_n 2^{(m - \alpha) n}$ for some integer $m > \alpha$ and every multi-index $k \in \N^d$ such that $| k | = m$. 
\end{enumerate} 
In particular, the functions $f_0 = \V_0 f$, $g_n = \W_n f$ verify $(2)$.
Moreover, the norms $\| f_0 \|_{L^p} + \| 2^{n \alpha} \| g_n \|_{L^p} \|_{\ell^q}$ and $\| \V_0 f \|_{L^p} + \| 2^{n \alpha} \| \W_n f \|_{L^p} \|_{\ell^q}$ are equivalent.
\end{thm}

A first consequence of this result is the fact that the Besov spaces defined in Section~\ref{s:tightcrit} are independent from the choice of a particular wavelet basis or multiresolution analysis.

\begin{prop}[Equivalence of multiresolution analyses] \label{a:p:equiv}
For any $\alpha \in \R$ and any positive integer $r$ such that $r > | \alpha |$, the norm $\| \cdot \|_{\Bb}$ of Definition~\ref{def:Besov:wave} does not depend on the given $r$-regular multiresolution analysis, i.e. every $r$-regular multiresolution analysis yields an equivalent norm. 
\end{prop}
\begin{proof}
Theorem~\ref{a:t:multimeyer} gives the equivalence of norms for $\alpha > 0$. 

For $\alpha < 0$, $p,q \in [1,\infty]$, define $\alpha' = - \alpha$, $1/p+1/p'=1$ and $1/q + 1/q' =1$. We introduce the following norm which is clearly independent from the choice of multiresolution analysis:
$$
\| f \|_{{ \tilde{\B}^{\alpha}_{p,q}}} = \sup_{\substack{g \in L^{p'} \\ \| g \|_{{\B^{\alpha'}_{p',q'}}} \leqslant 1 } } \langle f , g \rangle  
$$
(notice that this norm is slightly different from the norm of the dual of $\B^{\alpha'}_{p',q'}$, because we choose $\B^{\alpha'}_{p',q'}$ to be the complection of $C^\infty_c$ with respect to the norm $\| \cdot \|_{\B^{\alpha'}_{p',q'}}$ ).

We want to show that $\| \cdot \|_{ \tilde{\B}^{\alpha}_{p,q}}$ and $\| \cdot \|_{\Bb}$ are equivalent. Let $f \in \C^\infty_c$. 
The bound $\| f \|_{\tilde{\B}^{\alpha}_{p,q}} \lesssim \| f \|_{\B^{\alpha}_{p,q} }$ is straightforward: by Theorem~\ref{a:t:multimeyer}  we can write $g = \V_0 g + \sum_n \W_n g$ and obtain
$$
\langle f , g \rangle = \langle \V_0 f , \V_0 g \rangle + \sum_{n \geqslant 0} \langle \W_n f , \W_n g \rangle \leqslant \| f\|_{\Bb} \| g \|_{\B^{\alpha'}_{p',q'}}
$$
thanks to the orthogonality in $L^2$ between spaces $W_n$ and Hölder's inequality.

To show that $\| f \|_{\B^{\alpha}_{p,q} } \lesssim \| f \|_{\tilde{\B}^{\alpha}_{p,q}} $, recall that if $f \in L^p(\mu )$ then 
$$
\| f\|_{L^p (\mu)} = \sup_{g \in L^{p'} ( \mu ) , \| g \|_{L^{p'}} \leqslant 1 } \int f(x) g(x) \mu ( \mathrm{d}x)
$$ (see e.g. Lemma~1.2 of \cite{bcd}). Then for every $\delta >0$ there exists $h_0 \in L^{p'}$ such that $\| h_0 \|_{L^{p'}} \leqslant 1$ and  
$\| \V_0 f \|_{L^p} \leq \int \V_0 f (x) h_0 (x) \mathrm{d}x + \delta $. Take $Q_N^{q'}=\{ (a_n)_{n \geqslant 0} \in \ell^{q'} \mid \| a_n \|_{\ell^{q'}} \leqslant 1 \mbox{ , } a_n=0 \mbox{ for }n>N \}$. We have
$$
\| f \|_{\Bb} = \| \V_0 f \|_{L^p} +  \sup_{N \in \N} \sup_{(a_n) \in Q_N^{q'}} \sum_{n = 0}^N a_n 2^{\alpha n} \| \W_n f \|_{L^p}  .
$$
As above, for every $n \geqslant 0$ there exist $g_n \in L^{p'}$ such that $\| g_n \|_{L^{p'}} \leqslant 1$ and $\| \W_n f \|_{L^p} \leqslant \int \W_n f (x) g_n (x) \mathrm{d}x + \varepsilon_n . $ 
Now we can estimate the norm
\begin{align*}
\| f \|_{\Bb} & \leqslant \langle \V_0 f , \V_0 h_0 \rangle + \sup_{N \in \N} \sup_{(a_n) \in Q_N^{q'}} \sum_{n =0}^N  \langle \W_n f , 2^{n \alpha} a_n \W_n g_n \rangle + \varepsilon \\
\varepsilon  & = \delta + \sup_{N \in \N} \sup_{(a_n) \in Q_N^{q'}} \sum_{n =0}^N 2^{n \alpha} a_n \varepsilon_n
\end{align*}
where we used the fact that the spaces $W_n$ are orthogonal in $L^2$. The remainder $\varepsilon$ can be made arbitrarily small: indeed $\sum_{n =0}^N 2^{n \alpha} a_n \varepsilon_n \leqslant \| 2^{\alpha n}\|_{\ell^{q}} \sup_{n \geq 0} \varepsilon_n $ (recall that $\alpha < 0$).  Define 
$$
g_N = \V_0 h_0 + \sum_{n =0}^N 2^{n \alpha} a_n \W_n g_n.
$$ 
The operators $\V_n : L^p \rightarrow L^p$ and $\W_n : L^p \rightarrow L^p$ are uniformly bounded: we can estimate the norm of $g_N$ as
$$
\| g_N \|_{\B^{\alpha'}_{p',q'}} \leqslant \| h_0 \|_{L^{p'}} + \| 2^{n \alpha'} 2^{n \alpha} a_n \| g_n \|_{L^{p'}} \|_{\ell^{q'}} \leqslant C
$$
and then
$$
\| f \|_{\Bb} \leqslant \sup_{N \in \N} \sup_{(a_n) \in Q_N^{q'}} \langle f , g_N \rangle + \varepsilon = \sup_{\substack{g_N \in L^{p'} \\ \| g_N \|_{\B^{\alpha'}_{p',q'}} \leqslant C } } \langle f , g_N \rangle + \varepsilon \lesssim \| f \|_{\tilde{\B}^{\alpha}_{p,q}} + \varepsilon.
$$
This completes the proof of the result for $\al \neq 0$. The case $\al = 0$ can then be recovered by interpolation.
\end{proof}

We now introduce the Littlewood-Paley decomposition. We refer to \cite[Chapter~2]{bcd} for this definition.

\begin{prop}[Dyadic partition of unity]
 There exist $\chi \in \C^\infty_c (\R^d)$ with values in $[ 0, 1] $and support cointained in the
  ball $\mathcal{B}=\{ x \in \R^d \mid  | x| \leqslant 3/4 \}$, and $\rho \in C^\infty_c (\R^d)$ with values in $[ 0, 1]$ and support
  contained in the annulus $\mathcal{A}=\{ x \in \R^d \mid 3/4 \leqslant | x| \leqslant 8/3 \}$, such that for every $x \in
  \mathbbm{R}^d :$
  \[ 1 = \chi ( x) + \sum_{n \geqslant 0} \rho ( 2^{- n} x) \]
  and the sum is finite. We have also that, if $| n - n' | \geqslant 2$ :
  \begin{equation}
    \supp \rho ( 2^{- n} \cdot ) \cap \supp \rho ( 2^{- n'}
    \cdot ) = \emptyset 
  \end{equation}
  and if $n \geqslant 1$:
  \begin{equation*}
    \supp \chi \cap \supp \rho ( 2^{- n} \cdot ) = \emptyset
  \end{equation*}
\end{prop}

\begin{defi}[Littlewood-Paley-Besov space]\label{a:def:litpaley}
Let $f \in C^\infty_c$, for every $n \geqslant - 1$ the dyadic
  Littlewood-Paley blocks are defined as
  \begin{align*}
    \Delta_{- 1} u & =  \mathcal{F}^{- 1} ( \chi \hat{f}) \\
    \Delta_n u & =  \mathcal{F}^{- 1} ( \rho ( 2^{- n} \cdot )  \hat{f})  \quad \mbox{for every } n \geqslant 0
  \end{align*}
where $\mathcal{F}(f) = \hat{f}$ is the Fourier transform of $f$ (and $\mathcal{F}^{-1}$ its inverse). 

Define the norm $\| \cdot \|_{\B^{\alpha, \mathrm{LP}}_{p,q}}$ as
$$
\| f \|_{\B^{\alpha, \mathrm{LP}}_{p,q}} = \Ll\| \Ll(2^{\alpha n} \| \Delta_n f \|_{L^p}\Rr)_{n \ge -1} \Rr\|_{\ell^q}
$$
and the Littlewood-Paley-Besov space $\B^{\alpha, \mathrm{LP}}_{p,q}$ as the closure of $C^\infty_c$ with respect to this norm.
\end{defi}

\begin{rem}
It is easy to check that the space $\B^{\alpha, \mathrm{LP}}_{p,q}$ does not depend on the choice of a dyadic partition of unity $\chi, \rho \in C^\infty_c$, and that the operators $\Delta_n : L^p \rightarrow L^p$, $\Delta_{-1} : L^p \rightarrow L^p$ are uniformly bounded for every $p \in [1,\infty]$ (see \cite[Section~2.2]{bcd}).
\end{rem}

\begin{rem}[Equivalence of LP-wavelet Besov spaces]\label{a:r:litpaley-mult}
The space $\B^{\alpha, \mathrm{LP}}_{p,q}$ defined above coincides with the Besov space $\Bb$ that we used throughout these notes (Definition~\ref{def:Besov:wave}): i.e. for $f \in C^\infty_c$ their respective norms are equivalent. Indeed, the functions $\Delta_{-1} f$ and $\Delta_n f$ verify the conditions within point $(2)$ of Theorem~\ref{a:t:multimeyer}. The property
$$
\| \partial^{k} \Delta_n f \|_{L^p} \leqslant \varepsilon_n 2^{(m - \alpha) n}
$$
for $\varepsilon_n \in \ell^q$ 
is obtained by \emph{Bernstein estimates} \cite[Lemma~2.1]{bcd}, while the other two conditions are easily checked directly. 
\end{rem}

Now we can use Theorems~2.82 and 2.85 of \cite{bcd}, which yield a general proof of Proposition~\ref{p:mult-2}.

\begin{thm}[Multiplicative inequalities]
\label{a:t:mult}
Let $p, p_1, p_2, q, q_1, q_2 \in [1,\infty]$ be such that
$$
\frac{1}{p} = \frac{1}{p_1}  +   \frac{1}{p_2} \quad \text{ and } \quad \frac{1}{q} = \frac{1}{q_1}  +   \frac{1}{q_2}.
$$
(1) If $\al > 0$, then the mapping $(f,g) \mapsto fg$ extends to a bilinear continuous functional from $\B^\al_{p_1,q_1} \times \B^\al_{p_2,q_2}$ to $\Bb$.

\noindent (2) If $\al < 0 < \be$ with $\al + \be > 0$, then the mapping $(f,g) \mapsto fg$ extends to a bilinear continuous functional from $\B^\al_{p_1,q_1} \times \B^\be_{p_2,q_2}$ to $\Bb$.
\end{thm}
%

\begin{rem}\label{a:r:mult-2}
Theorem~\ref{a:t:mult} yields, as announced, a complete proof of Proposition~\ref{p:mult-2}. In Section~\ref{s:tightcrit} we proved that for any $\alpha < 0$ and $\chi \in C^{r_0}_c$, $r_0=-\lfloor \alpha \rfloor$, the mapping $f \mapsto \chi f$ extends to a continuos functional on $\C^\alpha$. This result can be extended to $\Bb$ observing that $C^{r_0}_c \subset \B^{r_0}_{\infty,\infty}$ (see \cite[Section~2.7]{bcd}) and applying the theorem above.
\end{rem}

%
%

\end{document}